\documentclass[a4paper,11pt]{amsart}

\usepackage{amssymb,amsbsy,amsmath,amsfonts,amssymb,amscd}
\usepackage{latexsym}
\usepackage{mathtools}
\usepackage{graphics}
\usepackage{color}
\usepackage{comment}
\input xy
\xyoption{all}

\theoremstyle{plain}
\newtheorem{thm}{Theorem}[section]

\newtheorem{lemma}[thm]{Lemma}
\newtheorem{corollary}[thm]{Corollary}
\newtheorem{proposition}[thm]{Proposition}
\theoremstyle{definition}
\newtheorem{remark}[thm]{Remark}

\newtheorem{defin}[thm]{Definition}

\newtheorem{assumption}[thm]{Assumption}

\newtheorem{example}[thm]{Example}

\newtheorem{question}[thm]{Question}

\numberwithin{equation}{section}

\newcommand{\sA}{{\mathcal A}}
\newcommand{\sB}{{\mathcal B}}
\newcommand{\sC}{{\mathcal C}}

\newcommand{\sE}{{\mathcal E}}
\newcommand{\sF}{{\mathcal F}}

\newcommand{\sH}{{\mathcal H}}
\newcommand{\sI}{{\mathcal I}}

\newcommand{\sK}{{\mathcal K}}
\newcommand{\sL}{{\mathcal L}}
\newcommand{\sM}{{\mathcal M}}

\newcommand{\sR}{{\mathcal R}}
\newcommand{\sS}{{\mathcal S}}

\newcommand{\sU}{{\mathcal U}}

\newcommand{\sX}{{\mathcal X}}

\newcommand{\Z}{{\mathbb Z}}
\newcommand{\PP}{\ensuremath{\mathbb{P}}}

\newcommand{\CC}{\ensuremath{\mathbb{C}}}

\newcommand{\ZZ}{\ensuremath{\mathbb{Z}}}

\newcommand{\hol}{\ensuremath{\mathcal{O}}}
\newenvironment{dedication}
        {\begin{quotation}\begin{center}\begin{em}}
        {\par\end{em}\end{center}\end{quotation}}

\newcommand\om{\omega}
\newcommand\la{\lambda}
\newcommand\s{\sigma}

\newcommand\al{\alpha}

\newcommand\De{\Delta}
\newcommand\ga{\gamma}
\newcommand\de{\delta}
\newcommand\e{\epsilon}

\newcommand{\Lam}{\Lambda}

\DeclareMathOperator{\Hom}{Hom}

\DeclareMathOperator{\rank}{rank}
\DeclareMathOperator{\Heis}{Heis}

\newcommand{\HH}{\ensuremath{\mathbb{H}}}

\newcommand{\ra}{\ensuremath{\rightarrow}}

\def\bea{\begin{eqnarray*}}
\def\eea{\end{eqnarray*}}

\newcommand\dual{\mathrel{\raise3pt\hbox{$\underline{\mathrm{\thinspace d
\thinspace}}$}}}
\newcommand\qe{\ifhmode\unskip\nobreak\fi\quad $\Box$}       
\newcommand\iso{\cong}

\def\BOX{\hfill\lower.5\baselineskip\hbox{$\Box$}}

\newtheorem{theo}{Theorem}[section]

\newtheorem{remarkk}[theo]{Remark}

\DeclareMathOperator{\Sing}{Sing}

\usepackage{hyperref}
\title [$p_g=q=2$
Surfaces ]{On the components of the Main Stream of the moduli space of surfaces 
of general type with $p_g=q=2$.}

\author{Massimiliano Alessandro, Fabrizio Catanese}
\address {Mathematisches Institut der Universit\"at Bayreuth\\
NW II,  Universit\"atsstr. 30\\
95447 Bayreuth}
\email{alessandro@dima.unige.it}
\email{fabrizio.catanese@uni-bayreuth.de}
\address{Universit\`a degli Studi di Genova, DIMA Dipartimento di Matematica, I-16146 Genova}
\address{Korea Institute for Advanced Study, Hoegiro 87, Seoul, 
133--722.}
\thanks{AMS Classification: 14J29, 14J10, 14D15, 14D06, 14J60, 14L30.\\
 }
 
\date{\today}

\begin{document}

\maketitle

\begin{dedication}
This article is dedicated to the memory of Alberto Collino.
\end{dedication}

\begin{abstract}
We give first an easy construction of surfaces with $p_g=q=2, K^2=5$ and Albanese map of degree $3$,
describing a  unirational irreducible connected component of the moduli space of surfaces of general type,  which 
we show to be the only one of the Main Stream  fulfilling the Gorenstein Assumption (see Assumption \ref{GorAss}) with these invariants.    
 We call it the family of CHPP surfaces, since it contains the family constructed by Chen and Hacon \cite{c-h}, 
 and coincides with the one considered by  Penegini and Polizzi \cite{pe-po3}.

 We also give an easy construction of an  irreducible   connected component of the moduli space of surfaces of general type 
with $p_g=q=2, K^2=6$ and Albanese map of degree $4$,  which we call the family of PP4 surfaces
  since it contains	 the family constructed in \cite{pe-po4}.

Finally, we answer a question posed by Hacon and Chen \cite{c-h}, via  three  families of surfaces with $p_g=q$
 whose   Tschirnhaus module has a kernel realization with quotient  a nontrivial  homogeneous bundle.
  Two families have $p_g=q=3$,   the third is a  new family of surfaces with $p_g=q=2, K^2=6$ and Albanese map of degree $3$.
\end{abstract}


\section*{Introduction.}

The main purpose of this article is to introduce a new and simple construction method for minimal surfaces 
of general type $S$ with  $p_g=q=2$.

Recall that surfaces of general type with  $p_g=q$ are  the surfaces with the lowest value $\chi =1$ 
of the invariant $\chi= \chi (\hol_S) = 1 - q + p_g$,
and that the cases 
 $p_g=q= 4$, $p_g=q=3$ have been classified in \cite{bea}, respectively in \cite{c-c-ml}, \cite{h-p}, \cite{pirola}.

In spite of many contributions, \cite{zuc},
 \cite{manetti}, \cite{c-ml},  \cite{c-h}, \cite{pe11}, \cite{pe-po3}, \cite{pe-po2},  \cite{pe-po4}, \cite{c-ml-p}, \cite{pe13},  \cite{pe-po17}, \cite{b-c-f}, \cite{rito}, \cite{c-f},   \cite{pi-po},    \cite{po-ri-ro},  \cite{pe-pi},
the case $p_g=q=2$ is still widely open.

We assume that the  Albanese morphism $\al : S \ra A : = Alb(S)$
is surjective, else $S$ has a fibration onto a curve of genus $g=2$, which makes $S$  isogenous to a product of curves
(see \cite{zuc}, \cite{pe11}).

 Observe that another possibility for $S$ to have an irrational pencil $f : S \ra B$ is that $A= Alb(S)$ contains an
elliptic curve $E$, and $ f $ is induced by the composition of $\al$ with the projection $ A \ra A/E$. 
There are irreducible components of the moduli space where $S$ always contains an irrational pencil,
and actually this happens for many  of the  known examples.
We shall mostly consider components where the Abelian surface $A$, for general $S$, does not contain such
an elliptic curve. More generally, we define:

\begin{defin}
A component $\mathfrak N$ of the Moduli Space of surfaces 
of general type with  $p_g=q=2$ is said to be of the {\bf Main Stream} if 

 (1) the	 Albanese map is surjective and 

(2) $$ \ \  \{ A = Alb(S) | \  [S] \in \mathfrak N\}$$

contains an open set in a Moduli Space of polarized Abelian surfaces.
\end{defin}

In the case we are considering, there are two basic numerical invariants for such a surface $S$:

\begin{itemize}
\item
the degree $d$ of  the Albanese map $\al : S \ra A$,
\item
 the Pfaffian $\de$ of a polarization $D$ determined on $A^{\vee}$ by  $\al$.
\end{itemize}

 In this context, the simplest example of a component of the main stream is provided 
by the	 double covers $S$ of a principally polarized Abelian surface $A$ branched on a divisor $\sB \in | 2 \Theta|$
with simple singularities.

These surfaces have $K^2_S = 4$ and, conversely, each surface with $p_g=q=2$ and $K^2_S = 4$
belongs to this family, see for instance \cite{c-ml-p}.

Moreover,  $K^2_S \geq  4$ by \cite{deb} or  \cite{pardini}, and, by the Bogomolov-Miyaoka-Yau inequality,
$ 4 \leq K^2_S \leq  9$, but the case $K^2_S =9$  is believed  not to occur \footnote { indeed  all the proposed proofs are up to now wrong (including also their amendments).}, while all values $K^2_S= 4, 5,6,7,8$ 
do in fact occur.  

We observe right away that no explicit upper bound is known for the degree $d$, but that  up to now we have
only examples with $ d=2,3,4,6$  (and only one case  for $d=6$, but already three for $d=4$).

Our new method is based on the geometrical understanding of a result due to Chen-Hacon \cite{c-h};
their work singles out the important hypothesis  on $S$ (which is however not necessarily deformation invariant),
namely, that the Albanese surface $A$ does not contain any elliptic curve.  This hypothesis is generically 
verified  if we deal with a component of the Main Stream.

Before we dwell into the description of their result, let us show the main feature of our construction.

Let $A'$ be an Abelian surface with a divisor $D$ yielding a polarization of type $(\de_1,\de_2)$ (hence with Pfaffian
$\de = \de_1\de_2$).

Then,  $V: = H^0(A', \hol_{A'}(D))$ is a $\de$ -dimensional vector space, and we consider as usual  the group of translations
$G: =  \sK(D)$ leaving the isomorphism class of $ \hol_{A'}(D)$ invariant (this is the kernel of $\Phi_{D} : A' \ra {A'}^{\vee}$,
see \cite{mumford}). 
Then, setting $ H_{D} :  = (\ZZ/ \de_1) \times  (\ZZ/ \de_2)$,  $G \cong  H_D^2$ (we shall often use the shorthand notation  $H : = H_D$)
and $V$ is an irreducible  representation of a finite Heisenberg group $\sH_D : = \Heis (H)$, called the Schr\"odinger representation.
This representation has the property that the centre $\sC$ of $\sH_D$ ($\subset \CC^*$) acts by scalar multiplication, 
and we observe moreover that $\sH_D / \sC \cong G,$ $A' /G = : A= (A')^{\vee}$.

Our method consists in  describing a smooth surface 
$$S' \subset \PP^{\de-1} \times A' = \PP (V) \times A' ,$$  which is $G$-invariant for  the $G$-action
of product type 
on $ \PP(V) \times A' $ (the action of $G$ on $\PP(V)$ being  induced  by   the action of the Heisenberg group $\sH_{D}$ on $V$). 

 Then,  we obtain $S $ as the free quotient $S : = S' / G$.
 
  In order to get a full component of the moduli space, we must also consider such  surfaces 
 $X' \subset \PP (V) \times A' $ which have at most rational double points as singularities,
 and let $S'$ be the minimal resolution of $X'$ ($S' = X'$ if $X'$ is smooth).

Since we want to convince the reader that our description is simple, here are the equations of $S'$ in two important cases:

\begin{itemize}
\item
CHPP surfaces: $ p_g=q=2, K^2=5, d=3, \de=2,$
 $$ S'  : =  S' ( \la) : = \{ x_1 (y_1^3 + \la y_1 y_2^2) + x_2  (y_2^3 + \la y_2 y_1^2) =0 \} \subset \PP^1 \times A' ,$$
where $ \la \in \CC$, and $\{x_1, x_2\}$ is a canonical  basis of $V = H^0(A', \hol_{A'}(D))$.
\item
PP4 surfaces: $ p_g=q=2, K^2=6, d=4, \de=3,$
$$ S'  : =  S' ( \mu) : = \{ \rank (M) \leq 1\}  \subset \PP^2 \times A' ,$$
\[
		M=\begin{pmatrix}
			x_1 &  x_3 & x_2 \\
			y_1^2 + \mu   y_2  y_ 3 & y_3^2 + \mu   y_1 y_2 &  y_2^2  + \mu y_1y_3
		\end{pmatrix}
		\]		
where $\{x_1, x_2, x_3\}$ is a canonical basis of $V = H^0(A', \hol_{A'}(D))$ and $\mu \in \CC$  (we shall later  alternatively use  the notation $\{s_1, s_2, s_3\}$ instead of $\{x_1, x_2, x_3\}$, in order to emphasize that these are sections of a line bundle on $A'$).
\end{itemize}

Our description is motivated by a result by Chen and Hacon \cite{c-h}: consider  more generally a surface $S$ with $p_g=q$ and
 a surjective morphism of degree $d$,  $\al : S \ra A$ onto an Abelian surface $A$, such that $\al$ does not factor
 through another Abelian surface.

 One defines the   Tschirnhaus Bundle  $\sE^{\vee}$ via the split exact sequence
$$ 0 \ra \hol_A \ra  \al_* (\hol_S) \ra \sE^{\vee} \ra 0.$$

By relative duality, we have the split exact sequence
$$ 0 \ra \om_A \cong \hol_A \ra \al_* (\om_S)\ra \mathfrak F \ra 0,$$
where $\mathfrak F$ is a subsheaf of $\sE$.

Theorem 3.5 of \cite{c-h}  states that, if $p_g=q=2$ and $A$  does not contain any elliptic curve, then 
we have an exact sequence

$$ 0 \ra \mathfrak H \ra \hat{L} \ra (-1_A)^* (\mathfrak F)\ra 0,$$
where $ \mathfrak H$ is a homogeneous bundle and $ \hat{L}$ is the Fourier-Mukai transform 
 of a    negative definite line bundle $L$  on $A':=A^\vee$ with $\sL: = \hol_{A'}(D) : = L^{-1}$  a polarization with Pfaffian  $\de$.

We have an isogeny $ \Phi_{D}  : A' \ra A' / \sK(D)  = {A'}^{\vee} = A$, and one main result
of the theory  
of Fourier-Mukai transforms says that  $$(-\Phi_{D})^*(\hat{L}) \cong \sL \otimes V^\vee$$ (cf. 
\cite[ formula (3.10)]{mukai},  see also \cite[Prop. 11.9]{polishchuk}).

Hence, if we take the fibre product of $S \ra A$ with  $\Phi_{D}: A' \ra A$, the pull-back $\mathfrak F'$ of $\mathfrak F$
is a quotient of $\sL \otimes V^\vee = \hol_{A'}(D)  \otimes V^\vee$ by a homogenous  bundle
$ \mathfrak H'$.

We have therefore the exact sequence

\begin{equation} \label{MainSequenceCH} 
	\ 0 \ra \mathfrak H' \ra \sL \otimes V^\vee   \ra  \mathfrak F'\ra 0,
\end{equation}

which is Heisenberg-equivariant (and indeed $\sK(D)$-equivariant).

The examples we have described above are just cases where the 
homogenous  bundle 
$ \mathfrak H'$ is zero, equivalently $ \mathfrak H$ is zero. Under this assumption $\mathfrak F$ is a vector bundle (locally free sheaf),
and then $ \mathfrak F \cong \sE$.

On the other hand, on page 227 of \cite{c-h}, it is asked whether the case $ \mathfrak H \neq 0$
can occur.

We give a positive answer, constructing two families of examples where $ \mathfrak H \neq 0$ and $p_g=q=3$,
 and one  example with $p_g=q=2$. 
Here the Albanese 
variety  $Alb(S)$ of $S$ admits a surjection onto an Abelian surface
$A$, and the composition of the Albanese map  $alb_S$ of $S$ with this surjection yields $ \al : S \ra A$
of degree $d=3$, respectively $d=4$, while  $\de$ equals respectively $3, 4$.

Also, the construction  of \cite{pi-po} implicitly provides another example with  $p_g=q=2$, $d=3$ (and we believe with $\de_1=\de_2=2$),
but in their case the construction is quite different and  not directly related to our simple method
since our Gorenstein Assumption (\ref{GorAss}) is not verified.

From our viewpoint, the exact sequence \eqref{MainSequenceCH} is interpreted as showing that $S' \subset \PP(V) \times A'$
is contained in a projective subbundle, as the  equations of $S'$ in the    following important cases clearly show:

\begin{itemize}
\item
$ p_g=q=3, K^2=6, d=3, \de=3,$
 $$ S'  : =  S' ( \la) :=  \{ (y, z) | \sum_j y_j x_j = \sum_j y_j^3 + \la y_1 y_2 y_3 =0  \} \subset \PP^2 \times A' ,$$
where $ y: = ( y_1, y_2,y_3) \in \PP^2,$ $\{x_1, x_2,x_3\} $ is a canonical  basis of $V = H^0(A', \hol_{A'}(D))$ and 
$\la \in \CC$ is such that $f_3(y): =   \sum_j y_j^3 + \la y_1 y_2 y_3 = 0$
defines a smooth  elliptic curve $C$; then, for general $\la$, $S'(\la)$ is smooth. 

Hence, $S' \subset C \times A'$ and $S$ has irregularity $ q=3$ since $G$ acts by translations on $C$.
\item
$ p_g=q=3, K^2=6, d=4, \de=4,$ with a polarization $D$ of type $(1,4)$,
$$ S'  : =  S' ( \la) :=  \{ (y, z) | \sum_j y_j x_j = Q_1(y) = Q_2(y)  =0  \} \subset \PP^3 \times A' ,$$
where $ y \in \PP^3,$ $\{x_1, x_2,x_3 , x_4\}$ is a basis of $V = H^0(A', \hol_{A'}(D))$ and 
$$Q_1(y) : = y_1^2 + y_3^2 + 2 \la y_2 y_4 , \ \ Q_2(y) : = y_2^2 + y_4^2 + 2 \la y_1 y_3, \ \la \neq 0, \pm 1, \pm i.$$
 The intersection of the two quadrics 
defines an elliptic curve $C$ of degree $4$, 
$$ C : =  \{ y ~ | ~  Q_1(y) = Q_2(y)  =0  \} \subset \PP^3 \}$$
on which $G = (\ZZ/4)^2$ acts by translations.\footnote{The case of a polarization of type $(2,2)$ cannot occur
since $G = (\ZZ/2)^4$ cannot act faithfully  on an elliptic curve.}

Again here  $S' \subset C \times A'$ and $S$ has irregularity $ q=3$ since $G$ acts by translations on $C$.

\item

AC3 surfaces: $ p_g=q=2, K^2=6, d=3, \de=3,$
 $$    S' : = \{ (y,z) \in \PP( V) \times A' | \sum_j y_j x_j(z) = 0 , \sum_i y_i^2 y_{i+1} =0\}, ~ S'  \subset C \times A'$$
where $ y: = ( y_1, y_2,y_3) \in \PP^2,$ $\{x_1, x_2,x_3\} $ is a canonical  basis of $V = H^0(A', \hol_{A'}(D))$,
$ C =  \{ y| \sum_i y_i^2 y_{i+1} =0\}$.

Here, $S' \subset C \times A'$ and $S$ has irregularity $ q=2$ since $G$ does not act by translations on $C$.

\end{itemize}

As the reader might have observed, the equations that we have shown in all the five examples
are either a cubic equation in the
variables $(y_j)$, or some quadratic equations. 

This is due to another main ingredient in our approach, namely the use of the theory by Casnati and Ekedahl
of  Gorenstein coverings of small degree $d=3,4,5$ \cite{c-e}, \cite{casnati5}.

The choice  to use this theory forces us to make a slightly restrictive assumption, which we now describe.

We have a surjective morphism $\al : S \ra A$, where $A$ is an Abelian surface, and $S$ is the minimal model of
a surface of general type. $\al$ is generically finite of degree $d \geq 2$, and the canonical divisor
 $K_S $ equals the ramification divisor
$R$ of $\al$. Any rational curve $C$ in $S$ is mapped to a point in $A$,  hence $\al$ factors through 
a morphism $ a: X \ra A$ of  the canonical model
$X$ of $S$, which is a Gorenstein surface.

If $a$ is a finite morphism, then we can directly apply the theory of Casnati and Ekedahl, implying
that $X$ embeds into $\PP(\sE^{\vee})$, and we can then use the structure theorems of \cite{c-e} \cite{casnati5}
for degree $ d \leq 5$.

In general, we can consider the Stein factorization $ S \ra X \ra Y \ra A$, where the last morphism $f : Y \ra A$
is finite of degree $d$, but $Y$ need not be Gorenstein. For this reason, most authors have used until now  more complicated formulae, due to Miranda and Hahn-Miranda, \cite{miranda},\cite{h-m}, describing $Y$ as $ Spec ( \hol_A \oplus \sE^{\vee})$ \footnote{In \cite{miranda} and \cite{h-m} $\sE^\vee$ is called $\sE$.}.

Still, restricting our attention to the open set $$A^0 = A \setminus \{ z ~ | ~  \dim ( a^{-1} (z) ) =1\},$$ 
we have a finite morphism $X^0 \ra A^0$, hence a rational map
$$ \psi : X \dashrightarrow \PP(\sE^{\vee}),$$
  with image $Z$ which is birational to $S$.
The natural question is: when is  $\psi$  indeed a morphism? For instance, is it so  when
$Z$ is normal?

At any rate, we  propose in the present article some  basic assumptions:

\begin{assumption} {\bf (Gorenstein Assumption)} \label{GorAss}
(I) We are given  a surjective morphism of degree $d \geq 3$, $\al : S \ra A$, where $A$ is an Abelian surface,  $S$ is the minimal model of
a surface of general type with $ q= p_g$, and $\al$ enjoys the property of the Albanese map, that it does not factor through a morphism
of $S$ to another Abelian surface. 

(II) We make the assumption that $\al$ induces
an embedding  $ \psi : X \ra \PP(\sE^{\vee})$ of the canonical model $X$ of $S$.
\end{assumption}

\begin{remark}
The Gorenstein assumption holds true if $ a : X \ra A$ is finite, but the example of CHPP surfaces shows that it may hold more generally without $a$ being finite.

\end{remark}

An alternative to the hypothesis of having a component of the Main Stream is the following

\begin{assumption} {\bf (Generality Assumption)} \label{GenAss}
We make here the same assumptions (I), (II) as in \ref{GorAss}, and we require moreover that:

 (III)  there exists a polarization $\sL:=\hol_{A'}(D)$ on $A' = A^{\vee}$ such that the pull back $\sE'$ of $\sE$ to $A' = A^{\vee}$  via $\Phi_D\colon A'\to A$ is a vector bundle fitting into an exact  sequence
\begin{equation} \label{MainSeqGenAss}
  \ 0 \ra \mathfrak H' \ra \sL \otimes V^\vee   \ra  \sE'\ra 0, \tag{$\bullet$}
\end{equation}
which is Heisenberg-equivariant ($\mathfrak{H}', V$ are as above:  $\mathfrak{H}'$ is a homogeneous vector bundle and $V:=H^0(A', \hol_{A'}(D))$,
and indeed the Heisenberg action on  $\sL$, respectively  $ V^\vee$,
makes the above sequence $\sK(D)$-equivariant).

Moreover, we consider $A$ endowed with the dual polarization (see for instance \cite[Sec. 14.4]{b-l} for the notion of dual polarization).

\end{assumption}

\bigskip

\begin{remark}
We shall mostly consider the case  $d \geq 3$, as we want to use the theory by Casnati-Ekedahl.

 For the case $d=2$, even if  $A$ does not contain any elliptic curve,
  the remark made on page 226 of  \cite{c-h}\footnote{observe moreover  that in the remark on page 226  there is an error of sign, 
it should  be $\hol_A(\Theta)$ instead of $\hol_A(- \Theta)$.} is wrong. In fact, the hypothesis $d=2$ does not imply that $\mathfrak{F}$ is a line bundle (yielding a principal polarization) as the existence of the two families with $p_g=q=2$, $K^2=8$, $d=2$ and $p_g=q=2$, $K^2=6$, $d=2$, constructed respectively in \cite{pe11} and \cite{pe-po2}, clearly show.

\end{remark}

\bigskip

\bigskip

These are the main results of the present article.

\bigskip

\begin{thm}

The  CHPP surfaces  yield a  unirational irreducible connected component of the moduli space of surfaces of general type, which is the unique component of the Main Stream   such that there is a surface in this component which fulfills the Gorenstein Assumption \ref{GorAss} and has $K_S^2 = 5$, $p_g(S)= q(S)=2$, and  Albanese map $\al : S \ra A = Alb(S)$ of degree $d=3$.  In particular, this component coincides with the component  constructed in \cite{pe-po3}.

\end{thm}

\begin{thm}
The   four dimensional family of PP4 surfaces of general type  yields an irreducible connected component of
 the moduli space of surfaces of general type with 
$ p_g=q=2, K^2=6, d=4, \de=3.$ 

\end{thm}

\begin{thm}
		All  minimal surfaces $S$ of general type with $p_g=q=2$, $K^2=6$, $\delta=3$, with Albanese map of degree $d=3$ and satisfying the Generality Assumption \ref{GenAss} belong to the 
		 family described in  Subsection \ref{new-family}. 
						 
		 Moreover,  under the Generality Assumption, the only other minimal surfaces $S$ of general type with $p_g=q$, $K^2=6$, with $\al :S  \ra A$ a surjective morphism
		 of degree $d=3$ onto an Abelian surface, are the surfaces with $p_g=q=3$ described in Subsection \ref{Hesse-cubic}.
	\end{thm}

We have a similar example with $p_g=q=3$, $K^2=6$, with $\al :S  \ra A$ a surjective morphism
		 of degree $d=4$ onto an Abelian surface, see the final Section \ref{d=de=4} (here there is a computer script still missing).

We end by summarizing the situation concerning the known irreducible components of the Main Stream for surfaces
of general type with $p_g=q=2$:

\begin{itemize}
\item $ K^2=4$: there is a unique irreducible connected component, of  STANDARD surfaces,   with $d=2, \de=1$, and branch curve in $|2 \Theta|$;
\item
$ K^2=5, d=3$: there is  a  unique irreducible connected component  fulfilling the Gorenstein Assumption \ref{GorAss},  the component of CHPP surfaces;

\item 
$ K^2=6, d=2$: there are only three  irreducible connected components, see \cite{pe-po2}\footnote{ their construction works,
in spite of the incorrect assertion that the elliptic singularity maps to a base point of the linear system $|D'|$, where
$\mathfrak F= \mathfrak M_p(D')$: indeed  $p$ is a base point of $|D' + Q_i|$,
where $Q_i$ is  a 2-torsion line bundle.};
\item
$ K^2=6, d=3$: there is the new component of AC3  surfaces (see \cite{ca-se} and
 Subsection \ref{new-family}); 
\item
$ K^2=6, d=4$: there is  the irreducible connected component of PP4 surfaces, which  contains the irreducible one constructed in \cite{pe-po4};
\item
$ K^2=7, d=3$: there is the component of PP7 surfaces, \cite{pi-po}.

\item 
	$K^2=8, d=2$:
there is an irreducible  connected component of dimension $3$  constructed by Penegini in \cite{pe11} as follows.  Let $f : D \ra C$ be an \'etale double cover of a curve $C$ of genus $2$,
and let $g : C \ra C$ be the covering involution. Then $S = (D \times D)/ \ZZ/4$
where the action is free and generated by  $(x,y) \mapsto (y, g(x))$. 

The Albanese surface is the Jacobian $Jac (C)$ and the Albanese map factors through $ [(x,y)] \mapsto f(x) + f(y) \in C^{(2)}$,
and then via the birational morphism $C^{(2)} \ra Jac(C)$. In this way the branch locus
$\sB$ of the Albanese map is a divisor $\sB \in | 4 \Theta|$ with a point $O$ of multiplicity 6, and the sheaf $\mathfrak F = \mathfrak M_{O}^2 (2 \Theta)$ (part of the facts we state here can  also be found in \cite{po-ri-ro}).

\end{itemize}

Finally, here are the other known irreducible components of the moduli space of surfaces of general type with $p_g=q=2$, which are not of the Main Stream: 

\begin{itemize}
	\item  $ K^2=4$: none, there is only the component of STANDARD surfaces;
	\item  $ K^2=5$: none;
	\item  $ K^2=6$: none;
	\item  $ K^2=7, d=2$: 
	there are 3 irreducible  components, all of dimension $2$, see \cite{pe-pi}. For every surface in them, the Albanese surface has a  non-simple polarization of type $(1,2)$  and the branch curve $\sB \in |2D|$
	has a singularity of type $(3,3)$;
	
	\item  $ K^2=8, d=2$:
	there are two complex-conjugate rigid minimal surfaces whose universal cover is not biholomorphic to the bidisk $\HH \times \HH$,
	 \cite{po-ri-ro}.
	 \item $ K^2=8, d=4,6,4$: there  are here $3$ connected components with $ K^2=8$, two of them of dimension $3$ and one of dimension $4$, constructed by Penegini in \cite{pe11};  these are surfaces isogenous to a product of unmixed type
		and not   of the Main Stream. In \cite{pe13} the author points out that for these families $d\leq 6$ is an upper bound for the degree $d$ of the Albanese map. 		
		But indeed, as we calculated by hand,  confirming  a personal communication by Penegini, the respective degrees are (using the order of Table 1 of 
		\cite{pe11}) $ d=4,6,4$.
		Moreover, the respective monodromy groups
		of the Albanese covering are $$(\ZZ/2)^2, \mathfrak S_3 \times 
		\mathfrak S_3, D_4.$$ 
\end{itemize}

\begin{remark}
		  The surfaces with $p_g=q=2$ constructed in \cite{b-c-f},
		as stated  in Prop. 4.11 ibidem, lie in the components described in \cite{pe-po2}.
		
	\end{remark}

	\begin{question}
		  Does the case $K_S^2 = 5, d=2$ occur?
	\end{question}

\section{Some basic notation}

1) Given a finite Abelian group $H$, the finite  Heisenberg group $ \Heis(H)$ is the central extension 
$$ 1 \ra \mu_n \ra \Heis(H) \ra H \times H^* \ra 1,$$
where $\mu_n \subset \CC^* $ is the group of $n$-th roots of $1$, $n$ is the exponent of $H$,
$ H^* : = \Hom (H , \CC^*)$, and $\Heis(H)$ is the group generated by the respective actions
of $h \in H$ on $ \CC^H$ given by translation, $(h\cdot f) (x) = f (x + h)$, and of $ \chi \in H^*$
given by multiplication with the character, $ (\chi\cdot f)(x) = f(x) \chi(x)$.

This representation of $\Heis(H)$ on $ \CC^H$ is called the {\bf Schr\"odinger representation}.

\bigskip

2) Given a locally free sheaf $\sE$ on a variety, we follow the topologists' notation and we let 
$$\PP(\sE) : = Proj (Sym(\sE^{\vee})) = : Proj (\sE^{\vee}),$$
where in general $Sym(\sE)$ denotes the symmetric algebra 
$$Sym(\sE): = \bigoplus _m Sym^m(\sE).$$
 Moreover, we shall often use the shorthand notation $S^m(\cdot):=Sym^m(\cdot)$.
\bigskip

3) $A'$ shall be an Abelian surface endowed with a divisor class $D$ yielding a polarization with elementary divisors $\de_1,\de_2$.

One has a surjective morphism (\cite{mumford}) 
\[
\begin{split}
\Phi_D : &A' \ra A'^{\vee} : = Pic^0(A') \\
&x \longmapsto t_x^* (D) - D,
\end{split}
\]
 and its kernel $G: =  \sK(D) \cong H^2$, where $ H := (\ZZ/ \de_1)\oplus (\ZZ/ \de_2)$, and the Heisenberg
group $\sH_D : = \Heis(H)$ is a group of isomorphisms of the line bundle $\sL:=\hol_{A'}(D)$ mapping onto $G$.

 $V: = H^0(A', \hol_{A'}(D))$ is isomorphic to the Schr\"odinger representation of $\sH_D$.

\bigskip

4) If $\de_1=1$, so that $H$ is a cyclic group $H = \ZZ/\de$, we shall also denote $\Heis(H) = \sH_D = : \sH_{\de}$.

\bigskip

5) Assume that $\al  : S \ra A$ is a generically finite morphism of degree $d$, where $A$ is smooth (in our case $A$ shall be an Abelian surface). Then we have a Stein factorization $ \al = f \circ \pi $, with $\pi : S \ra Y$ and $ f : Y \ra A$,
where $Y$ is normal and $f$ is finite (of degree $d$).

Clearly $\al_* (\hol_S) = f_*(\hol_Y)$, and since $f$ is finite (quasi-finite and proper) and $A$ is smooth,
then $\al_* (\hol_S) = f_*(\hol_Y)$ is locally free. We have an exact sequence
\begin{equation*}
	0 \ra \hol_{A} \ra \al_\ast(\hol_S) \to \sE^\vee \ra 0
\end{equation*}
 which splits in characteristic zero, i.e.
 $$\al_* (\hol_S) = \hol_A \oplus \sE^{\vee} ,$$
 where $\sE^{\vee}$ is called the Tschirnhaus bundle of $\al$ (and of $f$). Note that some people call it Tschirnhausen sheaf.
 
 By duality for a finite morphism (see \cite{hartshorne}, exercises 6.10, page 239, and 7.2, page 249) 
 $$ \sH om (f_*(\hol_Y), \om_A) = f_* \omega_Y = \omega_A \oplus (\sE \otimes \omega_A)$$
 where $\omega_Y$ is the dualizing sheaf of $Y$.
 
  In dimension $2$,  $\omega_Y$ equals the sheaf
 of Zariski's differentials; 
 since $S$ is a resolution of singularities of $Y$, we obtain then that $\al_* (\omega_S) \subset  f_* \omega_Y$,
 hence
 $$ \al_*   (\omega_S) =   \omega_A \oplus \mathfrak F, \qquad \mathfrak F \subset  \sE \otimes \omega_A.$$
 This formula, if $A$ is an Abelian surface,  simplifies to 
 $$ \al_*   (\omega_S) =   \omega_A \oplus \mathfrak F, \qquad  \mathfrak F \subset  \sE.$$
 If $Y$ has Rational Double Points as singularities, then $\al_* (\omega_S) =  f_* \omega_Y$
 and we have equality $\mathfrak F = \sE$.
 
 Moreover, $\sE/ \mathfrak F$ is supported on a finite set (contained in the image of  the singular
 points of $Y$, and of the points where the fibre of $S \ra Y$ is positive dimensional), hence if $\mathfrak F$ is locally free,
 then $ \mathfrak F = \sE$.

 \begin{remark} \label{loc-freeness}
 	We observe here that, assuming $A$ to be an Abelian surface, under the Gorenstein Assumption \ref{GorAss}  it holds true that $\mathfrak{F}=\sE$. 
 	The argument is as follows.

 	Since $\psi\colon X\hookrightarrow \PP:=\PP(\sE^\vee)$ is a closed subscheme of the $\PP^{d-2}$-bundle $p \colon \PP\to A$ and $\PP$ is smooth, Theorem 13.5 of \cite{lipman} ensures that the dualizing sheaf $\om_X$ is 
 	$$ \om_X=\sE xt^{d-2}_{\hol_\PP}(\psi_\ast \hol_X, \om_\PP).$$
 	
 	Then, recalling that $a=p\circ \psi$, we apply to the previous equality the direct image $p_\ast$, getting on the left-hand side   
 	$$a_\ast (\om_X)=\al_\ast(\om_S)=\om_A\oplus \mathfrak{F}=\hol_A\oplus \mathfrak{F},$$
 	while on the right-hand side
 	$$ p_\ast \bigg(\sE xt^{d-2}_{\hol_\PP}(\psi_\ast \hol_X, \om_\PP)\bigg)=\sE xt^{d-2}_{p}(\psi_\ast \hol_X, \om_{\PP}),$$
 	where $\sE xt^{d-2}_{p}$ stands for the $(d-2)$-th derived functor of $p_\ast \circ \sH om_{\hol_\PP}$ and the last equality follows from
 	the \emph{Grothendieck spectral sequence} with starting page 
 	$$ E_2^{p,q}:=R^p p_\ast(R^q \sH om_{\hol_\PP})$$
 	taking into account the vanishing 
 	$$ \sE xt^{q}_{\hol_\PP}(\psi_\ast \hol_X, \om_\PP)=0  \qquad \makebox{for}\quad q\neq d-2,$$
 	see \cite[pp. 200--201]{gelfand-manin}.

 	Recalling that $\om_\PP=\om_{\PP|A}$ because $A$ is an Abelian surface,  since $(d-2)$-th order duality holds for $p\colon \PP \to A$ (see \cite[Definition 10, Example 12]{kleiman}), there is the following isomorphism
 
 	\begin{equation*}
 		\begin{split}
 			&\sE xt^{d-2}_{p}(\psi_\ast \hol_X, \om_{\PP|A})\iso \sH om_{\hol_A}(p_\ast \psi_\ast(\hol_X), \hol_A)= (a_\ast(\hol_X))^\vee\\
 			&=(\al_\ast(\hol_S))^\vee=(\hol_A\oplus \sE^\vee)^\vee=\hol_{A}\oplus\sE.
 		\end{split}
 	\end{equation*}
	Finally, we have obtained
	$$ \hol_{A}\oplus \mathfrak{F}=\hol_{A}\oplus\sE,$$
	which clearly yields our assertion, i.e., $\mathfrak{F}=\sE$.

 \end{remark}

 \begin{remark}
 Arnaud Beauville has shown to the second author that by Grothendieck duality 
 (Theorem 11.1 of \cite{hartshorneduality}) one has the exact sequence
 
 $$ 0 \ra \al_* \omega_S \ra \sH om ( \al_* \hol_S, \hol_A) \ra \sE xt^2 (\sR^1\al_*  \hol_S,  \hol_A) \ra 0,$$
 hence that the cokernel 
 $$ \sE / \mathfrak{F} \cong \sE xt^2 (\sR^1\al_*  \hol_S,  \hol_A),$$
 hence $ \sE / \mathfrak{F} = 0$ if and only if $\sR^1\al_*  \hol_S =0 $.
  \end{remark}

 \bigskip

 6) Given an ample divisor $D$ on an Abelian surface $A'$, the linear system $|D|$ has no base points
 by the theorem of Lefschetz if $\de_1 \geq 2$.

 6 i) If $\de_1=1$ and  $\de \geq 3$ we first show 
 that it has no base points 
 if it has no fixed part; since the base-point locus $\Sigma$ is $G$-invariant, hence it has cardinality a multiple
 of $|G| = \de^2$, while $D^2 = 2 \de$. 
 
 Note that the system $|D|$ has no fixed part (see \cite[Lemma 10.1.1]{b-l})  unless the pair $(A',\hol_{A'} (D))$
 is isomorphic to a polarized product  of two elliptic curves, namely
 \begin{equation} \label{pol-prod-(1,d)} 
 	(A',\hol_{A'} (D))  \cong (E_1 ,\hol_{E_1} (D_1 )) \times (E_2 ,\hol_{E_2} (D_2 )), \tag{$\ast$}
 \end{equation}
 where $ \deg (D_1)=1$, $\deg(D_2)= \de_2$.
 
 Hence, we conclude in particular that for $\de_1=1$, $\de \geq 3$, $|D|$ has no base points
if  $A'$ does not contain any elliptic curve.

 6 ii) If $\de_1=1$ and  $\de =  2$, $D$ has no fixed part unless (see \cite{barth}) $A'$ is the polarized product of two elliptic
curves,
\begin{equation} \label{pol-prod-(1,2)}
	 (A', \hol_{A'}(D)) = (E_1, \hol_{E_1} (P_1)) \times (E_2, \hol_{E_2} (2 P_2)), \tag{$\ast \ast$}
\end{equation}
where $P_1, P_2$ are points;  in this  case the base locus 
equals the curve $\{P_1\} \times E_2$. 

If there is no curve in the base locus,
by $G$-invariance, the base locus consists of 4 distinct points.

 Hence, in all cases, given a basis $x_1, x_2$ of $H^0 (A',\hol_{A'} (D))$, at each base point
either  $x_1$ or $x_2$ is a local parameter.

\section{The basic construction for CHPP surfaces, having $p_g=q=2, K^2=5$ and Albanese map of degree $3$}

In this section  $A'$ is an Abelian surface with a divisor $D$ yielding a polarization of type $(1,2)$.

Then,  $V: = H^0(A', \hol_{A'}(D))$ is a two-dimensional vector space, and there is a group of translations
$G: =  \sK(D) \cong (\ZZ/2)^2$ leaving the isomorphism class of $  \hol_{A'}(D)$ invariant ($G$ is the kernel of $\Phi_D : A' \ra A:= (A')^{\vee}$,
see \cite{mumford}).

There are two generators $g_1, g_2$  acting on  $V: = H^0(A', \hol_{A'}(D))$ by transforming a suitable basis $x_1, x_2$ 
as follows:
$$ g_1(x_1) = x_1, \ g_1(x_2)= - x_2, \ g_2 (x_1) = x_2, \ g_2(x_2) = x_1.$$

The action of $g_1, g_2$ has the property that $ \ga: = g_1 g_2 g_1 g_2 $ acts by multiplication by $ - 1$,
hence $V$  is  the Schr\"odinger representation of the order $8$ Heisenberg 
group $\sH : = \sH_2$ with centre  $\ZZ/2$, such that $\sH \cong D_4$ and $\sH / \langle \ga\rangle= G$.

Let us call $W:=V^\vee$ the dual representation of $V$, which actually turns out to be isomorphic to $V$. Namely, $y_1,y_2$ being the dual basis of $x_1,x_2$, 
$$ g_1(y_1 ) =  y_1 , \ g_1( y_2) = - y_2 , \ g_2 (y_1) = y_2, \ g_2(y_2) = y_1,$$
and $W, V$ are the same  representation of the Heisenberg group $\sH$.

\medskip

The basic observation is that on  the tensor product $ V \otimes W$ we have an action of $G$, since 
the centre of $\sH$, generated by $\ga$, acts trivially. And $ V \otimes W$ contains   (up to constants) precisely 
one  invariant element,
namely $ x_1 y_1 + x_2 y_2$.

\bigskip

We define now an action of $G$ on $\PP^1 \times A'$, of product type, where $G$ acts on $ \PP^1 = \PP(V)$
via the previous action of $\sH$ on  $V$, whereas $G$ acts on $A'$ by translations.

Let $H$ be the hyperplane divisor on $\PP^1$. 

Then we consider the family of  divisors $X'$  in  $\PP^1 \times A'$ which belong to  the linear system 
$ | 3 H \boxtimes D|: = |p_1^*(3H) + p_2^*(D)|$  and which are  left invariant by the action of $G$.\footnote{In the sequel we shall be more sloppy and write $ 3 H + D : = 3 H \boxtimes D$.}

The general equation of such divisors in $ | 3 H \boxtimes D|$ is of the form 
$$ X' : = \{x_1 P(y_1, y_2) + x_2 Q(y_1, y_2)\},$$
with $P,Q$ homogeneous polynomials of degree $3$. 

$g_2$-invariance is equivalent to $Q(y_1, y_2)= \e P(y_2, y_1) , \e = \pm 1$: here the choice of $\e$
amounts to requiring the equation $f : = x_1 P(y_1, y_2) + x_2 Q(y_1, y_2)$  to be an $\e$-eigenvector for the action of $g_2$.

$g_1$-invariance is equivalent to 
$$x_1 P(y_1, y_2) + x_2  \e P(y_2, y_1) = 
\e' [x_1 P(y_1, - y_2) - x_2 \e P(-y_2, y_1)], \ \ \e' = \pm 1$$ (the choice of $\e'$
amounts to requiring the equation to be an $\e'$-eigenvector for the action of $g_1$).

We can write 
$$  P(y_1, - y_2) =  \e' P(y_1, y_2), \ \ \e' = \pm 1,$$
that is, either 
$$ P(y_1, - y_2) =  P(y_1, y_2), \quad  {\rm or } \quad  P(y_1, - y_2) =  - P(y_1, y_2) . $$
In the first case $P $ is a linear combination of $y_1^3 , y_1 y_2^2$, in the second case
a linear combination of $y_2^3 , y_1^2  y_2$.

\begin{remark}\label{exchange}
(i) One may  observe in an elementary way that the choice of $\e = -1$ reduces to the case $\e=1$ by replacing the basis element $x_2$ with $-x_2$.
\smallskip

(ii) The equation $ f \in H^0( \hol_{\PP^1 \times A'}( 3 H + D)) = Sym^3 (W) \otimes V$. Since $X' : = \{ f=0\}$ is $G$-invariant, 
follows that $f$ is an eigenvector for the $G$-action, with eigenvalue a character $\chi \in G^*$.

We can then take as new equation $ (f \otimes \chi ) \in Sym^3 (W) \otimes (V \otimes \chi) \cong Sym^3 (W) \otimes V,$
where the last isomorphism follows since $\sH$ has a unique irreducible representation of dimension $2$,
and $4$ of dimension $1$, corresponding to $G^* = \sH^* : = Hom (\sH , \CC^*)$.

Hence, by a suitable change of basis in $V$ we may always assume that not only $X'$ is $G$-invariant, but also
its equation $f$ is $G$-invariant.

\end{remark}

\begin{defin}
We define our extended CHPP surfaces $X$ as the quotient $X : = X' /G$
of a surface 
$$ X' : =  X' ( \la) : = \{ x_1 (y_1^3 + \la y_1 y_2^2) + x_2  (y_2^3 + \la y_2 y_1^2) =0 \} \subset \PP^1 \times A' = : Z,$$
where $ \la \in \CC.$ 

Whereas a CHPP surface shall be  the minimal resolution of singularities of one which has only rational double points as singularities. 
 By abuse of notation, we shall often call CHPP surfaces the canonical models, that is, the quotients $X:=X'/G$.

\end{defin}

\bigskip

Observe that, for $\la=0$, $X'$ is a Galois covering of $A'$ with group $(\ZZ / 3)$.

\bigskip

\begin{proposition}\label{SingCHPP}
 An extended CHPP  surface $X'$ is reducible if we are in the exceptional case \eqref{pol-prod-(1,2)} of 6 ii), that is, if   
$(A',D)$ is a polarized product  of elliptic curves.

Otherwise, an  extended CHPP  surface  is always normal, and  smooth for general $\la$ and  general $(A',D)$.

 $G$ acts freely on $X'$, and the CHPP surfaces have ample 
canonical divisor,
and invariants 
$$ K_{X'}^2 = 20, \ K_{X}^2 = 5, \ q(X') = q(X) = 2, \ p_g(X') = 5, \ p_g(X) = 2,$$
$$ \pi_1(X') \cong \pi_1(X) = \ZZ^4.$$
Their Albanese map has degree $3$.

The branch locus $\Delta$ of the Albanese map of $X'$ consists of 4 curves in the linear system $|D|$,
which are generally distinct (hence $\De$ has a 4-uple point at the points $ x_1=x_2=0$); for $\la=0$ instead
$\De$ consists of the two curves $ \{x_1=0\}, \ \{x_2=0\}$ counted with multiplicity 2.

\end{proposition}

\begin{proof}

i) If we are in the exceptional case \eqref{pol-prod-(1,2)}, where $x_1 = e_1 s_1,$ $x_2 = 	e_1 s_2$, $e_1$ is the pull-back of a section defining $P_1$ on $E_1$, while 
$s_1, s_2$ are pull-backs of a basis	of $H^0( \hol_{E_2}(2 P_2))$,
then over the curve $E_2'  : = \{P_1\} \times E_2 $ we have 
 $x_1= x_2=0$, hence $( \PP^1 \times E_2' ) \subset X'$, and $X'$ is reducible.

From now on we assume that we are not in case \eqref{pol-prod-(1,2)}, hence the equations $x_1=x_2=0$ define 4 points and
$x_1,x_2$ are local parameters for $A'$.

\medskip

ii) For $\la =0$, we get that the derivatives with respect to $y_1, y_2$ vanish only when $x_1 y_1= x_2 y_2=0$,
which implies that $x_1x_2=0$.

 Since  $(A',D)$ is not the exception \eqref{pol-prod-(1,2)}  of 6 ii),   for  $x_1=x_2=0$  the divisors $x_1 =0,x_2=0$ are smooth and they intersect transversally in $4$ points; hence
$x_1,x_2$ are local coordinates, and  the partial derivatives with respect to $x_1,x_2$ vanish only on $y_1= y_2=0$: but these equations define the empty set in $\PP^1$.

 If one only of $x_1,x_2$ vanishes, say $x_1=0$, then $y_2=0$ and we have a smooth point if the divisor $x_1=0$
is smooth: this happens for general $(A',D)$.

\bigskip

iii) Identify $H, D$ with their pull back on  $Z= \PP^1 \times A' $. 

Since $K_Z= - 2 H $, adjunction gives $K_{X'} =  (H + D) |_{X'}$, 
and $$ K_{X'}^2  = (3 H + D) (H + D)^2 = 5 H D^2 = 20.$$
$G$ acts freely on $A'$, hence also on $X'$, therefore $K_{X}^2 = 5$.

We have the exact cohomology sequence associated to the exact sequence
$$  0 \ra \hol_Z( - 2 H ) \ra \hol_Z(  H + D ) \ra \hol_{X'} (K_{X'}) \ra 0,$$
and since 
$$H^0  (\hol_Z( - 2 H ) ) = 0, \ h^1  (\hol_Z( - 2 H ) ) = 1, \ h^2  (\hol_Z( - 2 H ) ) = 2,$$
$$H^1(  \hol_Z(  H + D ) ) =0,  \ H^2(  \hol_Z(  H + D ) ) =0, \ h^0(\hol_Z(  H + D ) ) = 4,$$
follows that $$p_g(X') = 5,\ \  q(X') = h^1 ( \hol_{X'} (K_{X'})) = 2.$$

Since $G$ acts trivially on $H^0 (\Omega^1_{A'}) \cong H^0 (\Omega^1_{X'})$, follows that $q(X)=2$.
Finally, $G$ acts trivially on $ H^1  (\hol_Z( - 2 H ) )$, while, as remarked at the beginning,
$H^0(  \hol_Z(  H + D ) ) = V \otimes W$, hence $H^0(  \hol_Z(  H + D ) )^G$ has dimension $1$ and thus $p_g(X) = 2$.

The isomorphism $ \pi_1(X') \cong \pi_1(A')$ follows from the theorem of Lefschetz since
$X'$ is an ample divisor on $Z = \PP^1 \times A'$.

Finally,  $ \pi_1(X') \subset  \pi_1(X)$ is a normal subgroup of  index $4$, with quotient group $G$.

Set as usual  $A : = A' /G$. Then, $A$ is the Albanese variety of $X$, hence $\pi_1(A)$ is a quotient
of $\pi_1(X)$. But $ \pi_1(X') \cong \pi_1(A' )\subset \pi_1(A)$ has index $4$, 
hence $ \pi_1(X) \cong \pi_1(A )$.

\medskip

iv) In general, we ask when $X'$ has  Rational Double Points as singularities, for $\la \neq 0$.
 
To calculate the singular points we may use the Remark \ref{exchange}, and restrict to the equation
$ f = x_1 (y_1^3 + \la y_1 y_2^2) + x_2  (y_2^3 + \la y_2 y_1^2) $.

The partials with respect to $y_1$, respectively $y_2$, yield:
$$ \frac{\partial{f}}{\partial{y_1}} = x_1  ( 3 y_1^2 + \la y_2^2) + x_2 ( 2 \la y_1 y_2)=0,\frac{\partial{f}}{\partial{y_2}} =
x_1   ( 2 \la y_1 y_2) + x_2 ( 3 y_2^2 + \la y_1^2) =0.$$

If $x_1 $ vanishes, but $x_2$ does not, we have a singular point only if $y_1 y_2=0 = ( 3 y_2^2 + \la y_1^2) $,
but the two polynomials do not vanish simultaneously, hence we have no singular point.  Similarly if $x_2 $ vanishes, but $x_1$ does not.

\medskip

If both $x_1,x_2$ vanish,  the two partials with respect to the (local parameters) $x_1, x_2$
vanish if and only if 
$$(y_1^3 + \la y_1 y_2^2) =   (y_2^3 + \la y_2 y_1^2) =0 \ \iff \   (y_1^2 + \la  y_2^2) =   (y_2^2 + \la  y_1^2) =0   .$$

This may only occur for $\la = \pm 1$, and we get then exactly two singular points.

\medskip

If both $x_1,x_2$ do not vanish, then a necessary condition for a singular point (or a ramification point for $\al'$)
is that 
$$( 3 y_1^2 + \la y_2^2) ( 3 y_2^2 + \la y_1^2)  -  ( 2 \la y_1 y_2)^2=0 \ \iff \ y_1^4 + y_2^4 + \frac{1}{\la} (3 - \la^2) y_1^2 y_2^2 =0 .$$
This equation does not vanish for $y_1=0$, hence we write $y_1=1, y_2=z$, and we get the equation
\begin{equation} \label{eq-sing-CHPP}
	1 + z^4 + \frac{1}{\la} (3 - \la^2)  z^2 =0, \tag{$\ast \ast \ast$}
\end{equation}
 whose roots come in opposite pairs $z, -z$.

\medskip

At a singular point of $X'$ we have:
$$ f : = x_1 f_1(\la, z)  + x_2  f_2(\la, z) =0, \ \nabla x_1 ( f_1(\la, z)) + \nabla x_2(  f_2(\la, z)) = 0,$$
whence we get as second coordinate a singular point of the pencil  $|D|$, corresponding to
the point $( f_1(\la, z),  f_2(\la, z)) \in \PP^1$.

Now,  since we are not in the exceptional case \eqref{pol-prod-(1,2)}, by the Zeuthen-Segre formula follows that the pencil $|D|$ gives rise
to  at most such $ 12$ 
singular  points,
since the Euler number of the blow up of $A'$ equals $4$, and then $ 4 =  - 2 D^2 + \mu = -8 + \mu$,
hence we have $\mu=12$ singular fibres counted with multiplicity.

For each such value of  $(u_1, u_2)$ corresponding to a singular fibre
we get the equation $ u_2 f_1(\la, z)  -  u_1 f_2(\la, z) = 0 $, and substituting the four values
of $z$ gotten by \eqref{eq-sing-CHPP}, we get equations for the parameter $\la$ for which $X'$ is singular.

\bigskip

v)   We want to show that always   $X'$ has only finitely many singularities, hence $X'$ is always normal.

In fact, a fibre of $ \PP^1 \times A' \ra A'$ is contained in $X'$ if and only if $x_1=x_2=0$. But   $x_1, x_2$ 
are  local parameters, hence the whole fibre cannot be contained in the singular locus.

The above proof shows that, in the other cases where $x_1 \neq 0$ or $x_2 \neq 0$, 
we have always a finite number of 	singular points on $X'$.

\bigskip

vi) Finally, the discriminant of the projection of $X'$ to $A'$ equals 

\begin{equation}
\Delta : =  \det \left(\begin{matrix}3x_1 & 2\la x_2 & \la x_1 &0
\cr 0 & 3 x_1 &  2 \la x_2 & \la x_1 
\cr  \la x_2  & 2\la x_1 & 3 x_2  &0
\cr 0  & \la x_2 & 2 \la x_1 & 3 x_2
\end{matrix}\right).
\end{equation}

Since $\Delta$ is given by the vanishing of a homogeneus polynomial of degree $4$ in $(x_1,x_2)$ we get, for each $\la$,  a product of 4 linear factors, hence the discriminant consists of 4 curves 
in the linear system $|D|$, counted with multiplicity.

For $\la=0$, we get $ 81 x_1^2 x_2^2=0$, which is of course expected since then we have a Galois covering with cyclic
Galois group of order 3.

\end{proof}

\begin{remark}
The morphism $\al : X \ra A$ never yields a Galois extension of function fields.
\end{remark}

The reason is: if $\al $ is Galois, then also the fibre product $X' \ra A$ is Galois, hence $X' \ra A'$
is Galois and  the equation of $X'$ is
$$X' = \{ x_1 y_1^3 + x_2 y_2^3=0\}.$$
The group $\mu_3$ of third roots of unity acts by 
$$ y_1 \mapsto y_1, \ \ y_2 \mapsto \e y_2, \ \ \e^3=1.$$

We claim that $G, \mu_3$ generate a group $G'$ of order $12$.

 Indeed, we see right away that $g_1$ and $\e$ commute,
while $$g_2 \e g_2 (y_1, y_2) = (\e y_1, y_2) = (y_1, \e^{-1} y_2) \ \iff \ g_2 \e g_2  =  \e^{-1}.$$
Hence $g_2$ and $\mu_3$ generate $\mathfrak S_3$, and 
$$G' = \mathfrak S_3 \times \ZZ/2, \quad  \ZZ/2 = \{ 0, g_2\}.$$

Since $X$ corresponds to the intermediate subgroup $ G < G'$ which is not normal ($G \cong \ZZ/2 \times  \ZZ/2 $), $\al$ is not Galois,
a contradiction.

\section{Deformations of the CHPP surfaces $X$}

We have constructed an irreducible  $4$-dimensional family (three parameters for the Abelian  surface $A'$,
and $\la$ as fourth parameter) of CHPP surfaces $X$, and we want to see that this yields
a  component of the moduli space of surfaces of general type.

In order to achieve this goal, it suffices to analyze deformations $\sX \ra T$ 
with connected base.

There are two guiding principles, coming from topology:

I) every deformation of $X$ comes together with a deformation of $X'$ preserving the $G$-action
(up to an automorphism of $G$),

II) every deformation of $X$, respectively of $X'$, comes together with a deformation of their Albanese maps $\al ' : X' \ra A'$,
$\al  : X \ra A$ which are generically finite of degree $3$; indeed any other surface homotopically equivalent
to $X$, resp. $X'$, has an Albanese map of degree $3$.

Taking the Stein factorization of the Albanese maps, we get  finite triple coverings $Y (t)  \ra A (t) $,
$Y(t ): = Spec (\al (t)_* (\hol_{X_t}))$, and similarly for the deformations of $X'$.

We observe that for our  surfaces $X'$ we have 
the exact sequence 
$$ 0 \ra \hol_Z(-3H - D) \ra \hol_Z \ra \hol_{X'} \ra 0 ,$$
whence by direct image the exact sequence
$$ 0 \ra \hol_{A'} \ra \al'_* (\hol_{X'}) \ra \hol_{A'}(-D)^{\oplus 2} \ra 0,$$
and the so-called Tschirnhaus bundle $(\sE')^{\vee}$ of the degree 3 map equals  $ (\sE')^{\vee} = \hol_{A'}(-D)^{\oplus 2}$.

Moreover, for small deformations, we shall have a composite morphism 
$$  X'_t \ra Y'(t) \ra \PP ( \al' (t)_* (\hol_{X'_t})/ \hol_{A'(t)}),$$
which is a $\PP^1$-bundle over $A'(t)$.

The deformations of $X'$ turn out to be more complicated to describe than the ones of $X$,
since the $\PP^1$-bundle can admit nontrivial deformations, as $X'$ deforms.

But the situation for $X$ is simpler.

\begin{lemma}
For every deformation $X_t$ of $X$, the Albanese map of $X'_t$ factors through a birational morphism
into $\PP^1 \times  A'(t)$. 

\end{lemma}
\begin{proof}
Any deformation of $X$ yields, as we already observed, 
a deformation of $X'$ which preserves the  $G$-action.

This implies  that the Tschirnhaus bundle $(\sE')^{\vee}$ splits according to the two eigensheaves for $g_1$,
and since they shall be exchanged by $g_2$, we have that $(\sE')^{\vee}$  is  always a direct sum of two copies of the same line bundle.
Which, of course, is a deformation of  $ \hol_{A'} (-D)$, hence it is this bundle up to translation on $A'(t)$.

\end{proof}

\begin{corollary}
Any small deformation of $X$ yields an embedding $ X'_t \subset \PP^1 \times  A'(t)$. The divisor class of
$X'_t$ is the class $ 3 H \boxtimes D_t$, where $D_t$ is a polarization of type $(1,2)$ on $A'(t)$.
\end{corollary}

The previous results allow us  to conclude that all small deformations of $X$ are given by deformations of $X'$ as hypersurfaces
inside  a threefold $\PP^1 \times  A'(t)$, where $A'(t)$ is a deformation of $A'$, and the action of $G$ is preserved;
hence every deformation of $X$ comes from a $G$-invariant deformation of $X'$, and we conclude that our families are locally complete.

\medskip

We want to show more:

\begin{thm}\label{CHPP-large}
Every deformation in the large of a CHPP surface is a CHPP surface.
\end{thm}
\begin{proof}
As well known (see for instance \cite{bc}, pages 625-626), it suffices to show that if we have a 1-parameter family
$X_t, t \in T$, where $T$ is a smooth curve,  which is a deformation in the large of a CHPP surface $X$,
then all the surfaces  $X_t$ are CHPP surfaces.

Under the above assumption  $X'_t$ is a deformation of $X'$, and we have
a birational morphism $X'_t \ra \PP^1 \times A'(t)$, whose image is a divisor $\Sigma_t$ in a linear system
$ |3 H \boxtimes D_t|$, where $D_t$ is a polarization of type $(1,2)$ on $A'(t)$.

The dualizing sheaf $\omega_{\Sigma_t}$ is the restriction of the invertible sheaf $\hol_{Z(t)}( H + D_t)$,
and it has $h^0( \omega_{\Sigma_t}) = 5 = p_g (X'_t )$.

Let $S'_t$ be the minimal model of $X'_t$.

Since $S'_t \ra \Sigma_t$ is a resolution of singularities, we see now that there are no conditions of subadjunction,
nor of adjunction.

 $\Sigma_t$ yields an extended CHPP surface and, since  $\Sigma_t$ is irreducible,  by Proposition \ref{SingCHPP} 
  we are not in the exceptional case \eqref{pol-prod-(1,2)} and  $\Sigma_t$ is normal.

If $\Sigma_t$ is normal and does not have rational double points as singularities,
then $ K_{X'_t}$ is the pull back of $(H + D_t)$ minus a non zero effective exceptional divisor,
hence   $ K_{X'_t}^2 < 20$, a contradiction.

\end{proof}

In the next subsection, which can be seen as a longer digression and is not essential 
for the results of  this article,  we consider the more difficult question of studying the deformations of 
the surfaces $X'$. 

\subsection{The deformations of the surface  $X'$}

This subsection is sort of a digression, and we want here to  look  at the deformations of $X'$:
hence we look at the cohomology group $H^1 (X' , \Theta_{X'})$
and the Kodaira-Spencer map.

The first isomorphism that we observe is 
$$\Theta_Z \cong  \hol_Z(2H) \oplus  \hol_Z^2 .$$

Then, we consider the exact sequence 
$$ 0 \ra  \Theta_{X'} \ra \hol_{X'} (2H) \oplus  \hol_{X'}^2 \ra \hol_{X'} (3 H + D) \ra 0,$$
and, since $X'$ is of general type,  $H^0 ( \Theta_{X'} )=0$, 
while  $H^0 (\hol_{X'} (2H) \oplus  \hol_{X'}^2 )$ has dimension $5$, and
$ H^0 (\hol_{X'} (3 H + D) ) $ has dimension $9=8-1+2$,
since it fits into the exact sequence 
$$ 0 \ra H^0 (\hol_Z) \ra H^0 (\hol_Z (3 H + D) ) \ra  H^0 (\hol_{X'} (3 H + D) ) \ra H^1 (\hol_Z) \ra 0 .$$
Finally,  since $\hol_{X'} (3 H + D)= \hol_{X'} (2 H + K_{X'})$
has vanishing second   cohomology group, and first of dimension 1,
we have  the exact cohomology sequence
$$ 0 \ra H^0 (\hol_{X'} (2H) \oplus  \hol_{X'}^2 ) \ra H^0 (\hol_{X'} (3 H + D) ) \ra H^1( \Theta_{X'})  \ra H^1 ( \hol_{X'} (2H) \oplus  \hol_{X'}^2 ) \ra$$
$$ \ra H^1 (\hol_{X'} (3 H + D) ) \ra H^2( \Theta_{X'})  \ra H^2 ( \hol_{X'} (2H) \oplus  \hol_{X'}^2 ) \ra 0,$$ 
and  since,  by the next lemma,  $H^i ( \hol_{X'} (2H))$ has dimension 6 for $i=1$, 3 for $i=2$, we get that $ H^1( \Theta_{X'}) $ has dimension
at most $14$, while $ H^2( \Theta_{X'}) $ has dimension either $13$ or $14$.

\medskip

Since however $ 10 \chi(X') - 2 K^2_{X'} = 0$, $H^1( \Theta_{X'}), H^2( \Theta_{X'})$ have the same dimension.

\begin{lemma}
$H^1 ( \hol_{X'} (2H)) $ has dimension 6, $H^2 ( \hol_{X'} (2H)) $ has dimension 3.
\end{lemma}
\begin{proof}
We use the exact sequence
$$ 0 \ra \hol_Z (-H - D) \ra 
\hol_Z(2H) \ra  \hol_{X'} (2H) \ra 0,$$
and the fact that by the K\"unneth formula $\hol_Z (-H - D)$ has all  cohomology groups vanishing,
hence 
$$  H^1 ( \hol_{X'} (2H)) \cong H^1(  \hol_Z(2H) ) = H^0(\hol_{\PP^1}(2)) \otimes H^1 (\hol_{A'}) $$
 has dimension 6,
while 
$$  H^2 ( \hol_{X'} (2H)) \cong H^2(  \hol_Z(2H) ) = H^0(\hol_{\PP^1}(2)) \otimes H^2 (\hol_{A'}) $$
has dimension 3.

\end{proof}

We  observe that  the image of $ H^1( \Theta_{X'}) $ inside $H^1 ( \hol_{X'})^2$ corresponds to the deformations
of $A'$ as a complex torus, but each deformation of $X'$ yields a deformation of $A'$ as an Abelian surface.

The deformations of $X'$  contains a family of dimension $ 3 + 7 -3= 7$ if we keep a hypersurface in $\PP^1 \times A'$,
but indeed deforming $X'$ we could  take a deformation of the trivial rank $2$ bundle.

The tangent space to the deformations of the trivial rank 2 bundle on $A'$ is given by
the vector space ${\rm Ext}^1 (\hol_{A'}^2, \hol_{A'}^2) \cong H^1 (\hol_{A'}^4)$, which has dimension $8$, but since we are interested in the deformations
of the associated projective bundle, we get a vector space of dimension $6$,
corresponding to the deformations with trivial determinant, $H^1 (End^0(\hol_{A'}^2))$
($End^0$ denotes as usual the space of trace zero endomorphisms): this is the explanation of the
map to the 6-dimensional vector space 
$  H^1 ( \hol_{X'} (2H)) \cong H^1(  \hol_Z(2H) )$.

\section{Unirational  moduli space for CHPP surfaces and its characterization}

We see immediately that the connected component of the moduli space that we have constructed is  unirational, since  we have a rational parametrization of the family of genus 2 curves $C$, endowed with a 2-torsion divisor in $Pic(C)$: it suffices to take a family of 6 ordered points in $\PP^1$ 
and the corresponding double coverings of $\PP^1$ branched on the said points.

As a consequence of the considerations made in the previous section, we have shown the following result, which constitutes a shorter proof of a theorem by Penegini and Polizzi \cite{pe-po3}:

\begin{thm}\label{mainthm}

The  unirational irreducible connected component corresponding to the CHPP surfaces is  the unique  component of the Main Stream  such that there is a  surface in this component which fulfills  the Gorenstein Assumption \ref{GorAss}  and has  $K_S^2 = 5$, $p_g(S)= q(S)=2$,
and  Albanese map $\al : S \ra A = Alb(S)$ of degree $d=3$.
 In particular, this component coincides with the component  constructed in \cite{pe-po3}.
\end{thm}
\begin{remark}\label{III}
In the above Theorem, one can remove the assumption `of the Main Stream' and
replace the Gorenstein Assumption by the Generality Assumption,
and indeed just property (III) with $ \mathfrak H' =0$.
\end{remark}

\begin{proof}
We have the  isogeny $ \Phi_D : A' \ra A' / \sK(D) = A' / G $, and 
if we have a component  $\mathfrak N$ of the Main Stream, there is some surface $S$ such that $A = Alb(S)$ contains no elliptic curve.

  Under this condition, by Theorem 3.5 of \cite{c-h}, we have an exact sequence for the pull back $\mathfrak F'$ of $\mathfrak F$

$$ 0 \ra \mathfrak H'  \ra \sL \otimes V^{\vee}  \ra \mathfrak F'\ra 0,$$
where $ \mathfrak H'$ is a homogeneous bundle and  $\sL $  a polarization 
of Pfaffian $\de = \  \rank ( \mathfrak H') + 2$.

We consider now those surfaces $S$ which satisfy the Gorenstein Assumption \ref{GorAss}: then  $\mathfrak{F}=\sE$ (see Remark \ref{loc-freeness}), and we have  $ K_S^2 = 3 + \de$  (see Proposition \ref{K^2}). This implies immediately that  $\de=2$ and therefore $\mathfrak{H}'=0$.

Hence, we get that $ D$ yields  a polarization of type $(1,2)$.

Taking the pull back $S'$ to $A'$ of the surfaces $S$, we find that  on an open set
of $\mathfrak N$, by the theory by Miranda-Casnati-Ekedahl, we get a section of 
$$ S^3(\sE') \otimes \det(\sE')^{-1} = S^3 (V^{\vee} ) \otimes \hol_{A'}(D),$$
 an equation defining an extended  CHPP surface $Z'$.
 For the surfaces $S$ satisfying the Gorenstein assumption we have $Z' = X'$, the pull back of the canonical model $X$. Hence the latter  surfaces $S$ are CHPP surfaces 
 and we conclude by Theorem \ref{CHPP-large}
 that $\mathfrak N$ is the connected component of the CHPP surfaces.

Finally, thanks to Proposition 6.1 of \cite{pe-po3}, we see immediately that this  component $\mathfrak{N}$ must coincide with the one constructed by Penegini-Polizzi in \cite{pe-po3}.

\end{proof}

\begin{remark}
If (III) of the Generality Assumption holds with $\mathfrak H'=0$ for some surfaces of the component $\mathfrak N$, then  their  canonical models $X$ 
have a pull-back $X'$ with a birational map $\psi : X' \dashrightarrow \
Z'$, where $ Z' \subset A' \times \PP^1$ is an extended CHPP surface with dualizing sheaf $\omega_{Z'} = \hol_{Z'} (H + D)$.

$\psi$ is an isomorphism where $X' \ra A'$ is finite, that is, outside a finite	 number of fibres
of $ \PP^1 \times A' \ra A'$. Where $Z'$ does not contain such a fibre, $Z' \ra A'$ is finite and
$Z'$ coincides with  the Stein factorization, so $\psi$ is a morphism there.

There remain the fibres which are contained in $Z'$, and for which $X' \ra A'$ has a positive dimensional fibre.

 By Proposition \ref{SingCHPP}, follows 
  that  $Z'$ is normal.

Take a blow up of $S'$, say $S^*$, such that $\psi$ becomes a birational
morphism on $S^*$.

Use the same notation for a divisor and its pull back to $S^*$.

By adjunction we have $K_{S^*} = K_{Z'} - \sA$, where $ \sA$ is the adjoint divisor (an effective divisor).

Similarly, $K_{S^*} = K_{S'} + E$, for some effective exceptional divisor $E$.

The conclusion is that

$$ K_{Z'} = K_{S'} + E + \sA .$$

Since $K_{S'}^2 = K_{X'}^2 = K_{Z'}^2 = 20$, and $E + \sA$ has negative self intersection
(alternatively, use that $K_{Z'} , K_{S'}$ are nef and big, and uniqueness of the Zariski 
decomposition)
we conclude  then that $ E + \sA = 0$, which means that $K_{Z'} $ pulls back to $ K_{X'}$.
Since $ K_{X'}$ is ample, it follows that  $X' = Z'$ (else there is a curve $C$ in $X'$ which
is contracted to a point in $Z'$, a contradiction), and $X$ is a CHPP surface. 

In conclusion, in view of Theorem \ref{CHPP-large} we have proven the statement of Remark \ref{III}.

\end{remark}

\section{The basic construction for PP4 surfaces, having $p_g=q=2, K^2=6$ and Albanese map of degree $4$}

Let here $A'$ be an abelian surface with a divisor $D$ yielding a polarization of type $(1,3)$. Set $\sL : =  \hol_{A'}(D)$.
	
Then, $V:=H^0(A', \hol_{A'}(D))$ is a three dimensional vector space and there is a group of translations $G:={\sK}(D)\iso (\Z/3\Z)^2$ leaving the isomorphism class of $\hol_{A'}(D)$ invariant. $V$ is the Schr\"odinger representation of the Heisenberg group $\sH_3$ associated
 to $H = \Z/3\Z$.
	
We are going to describe, using the method of Casnati-Ekedahl \cite{c-e}, some degree $4$ coverings $S'$ of $A'$ which
are invariant under the action of the group $G$.

 We shall call the quotients $S := S' /G$   PP4 surfaces, since our family coincides generically with the family constructed by Penegini and Polizzi in \cite{pe-po4}.

As usual, we set $ \sE' : = V^{\vee} \otimes  \hol_{A'}(D) =  V^{\vee} \otimes  \sL$ and we 
shall consider, according to Casnati-Ekedahl \cite{c-e}, 
 a rank two subbundle 
\[
\sF \to S^2(\sE')= (\sL^{\otimes 2})\otimes  S^2 ( V^{\vee} ) =  6 {\sL}^{\otimes 2}= \bigwedge^2(3\sL)\oplus \bigwedge^2(3\sL).
\]

By Casnati-Ekedahl we must have  either $ \det(\sF) = \det (\sE')= 3 D $,
and  $$c_2(\sF) = \de^2 \ K_S^2   -2 c_1(\sE')^2 + 4 c_2(\sE') = 9  ( \ K_S^2 - 4),$$
hence for $K_S^2 =6$ we must have $c_2(\sF)= 18$ (whereas for $K_S^2 =5$ we would 
have $c_2(\sF)= 9$).

We observe that 
$$
h^0(\sL)=3,  \quad D^2=6$$

and    $|D|$ has  no base points if $(A',D)$ is not a product polarization (see Section 1, \eqref{pol-prod-(1,d)}).
We shall make this assumption from now on.

We define the sheaf $\sF$ as the  cokernel of the natural homomorphism 
$$\hol_{A'} \ra V^{\vee} \otimes \hol_{A'}(D)$$
given by the natural invariant  $s : =  \sum_j y_j s_j \in  V^{\vee} \otimes V$ where $\{s_1,s_2,s_3\}$ is a basis of $V= H^0(\sL)$
and $\{y_1,y_2,y_3\}$ is the dual  basis of $ V^{\vee}$. Namely, we have
\begin{equation}
	\xymatrixcolsep{2pc}
	\xymatrix{
		0 \ar[r] &
	\hol_{A'}
	\ar[r]^-{s=
		\begin{pmatrix}
			s_1 \\
			s_2 \\
			s_3
	\end{pmatrix}} &
	V^{\vee} \otimes \hol_{A'}(D) = 3\hol_{A'}(D) \ar[r] & \sF \ar[r] & 0
}.
\end{equation}
  Since  $H^0(\sL)$ has no base points, $\sF$ is a
rank 2 bundle.

We shall soon show that  $\sF$ is  the image of the map $\wedge s$ in the  Koszul complex corresponding  to the section $s=\sum_j y_j s_j \in V^{\vee} \otimes H^0({\sL}).$

What we are doing is to consider the Koszul  complex associated to  the section $s= {}^t(s_1,s_2,s_3)\in H^0(\sE')$ (see \cite[Appendix B.3.4]{fulton}). Namely, defining $Z(s)$ to be the zero subscheme of $s$, we consider the sequence
\begin{equation}\label{koszulcomplex}
	0 \ra \bigwedge^3 {\sE'}^\vee  \ra \bigwedge^2 {\sE'}^\vee \ra\bigwedge^1 {\sE'}^\vee={\sE'}^\vee \ra  \sI_{Z(s)} \ra 0.
\end{equation}
which is exact on $A'=A'\setminus Z(s)$.

Note that $Z(s)=\emptyset$ if we assume that  $s_1,s_2,s_3$ have no common zeroes, as does in our case occur
for a general choice of $(A', D)$.

Then, since $\sE'=V^\vee \otimes \hol_{A'}(D)$, we get, in a simpler notation, the following exact sequence
\begin{equation}\label{koszulcomplex}
	0 \to \bigg( \bigwedge^3 V \bigg)\otimes \hol_{A'}(-3D)  \ra \bigg(\bigwedge^2 V\bigg) \otimes \hol_{A'}(-2D) \ra V \otimes \hol_{A'}(-D)=3\hol_{A'}(-D) \ra \hol_{A'} \ra 0.
\end{equation}

Dualizing the previous sequence, we get
\begin{equation}
	0 \to \hol_{A'} \overset{s}{\ra} V^\vee \otimes \hol_{A'}(D)=3\hol_{A'}(D) \overset{\wedge s}{\ra} \bigg(\bigwedge^2 V^\vee \bigg) \otimes \hol_{A'}(2D) \ra  \bigg( \bigwedge^3 V^\vee \bigg)\otimes \hol_{A'}(3D) \to 0
\end{equation}

Hence, 
\begin{equation}
	 \sF \overset{\wedge s}{\hookrightarrow}  \bigg(\bigwedge^2 V^{\vee}\bigg) \otimes \hol_{A'}(2D)=3\hol_{A'}(2D).
\end{equation}

We define  the quadruple cover $S'\to A'$ as given by the zero locus of a section
\[
\eta \in H^0(A',  \mathcal{F}^\vee \otimes S^2(\mathcal{E}')).
\]
This means that the fibres of $S'\to A'$ are the intersections of two conics in the $\PP^2$-bundle $Proj(\sE')$  (see \cite{c-e}).

We describe now  a 2-dimensional family of sections $\eta_{\la,\mu}: = \la i_1 \oplus \mu i_2$ depending on two complex parameters $\la, \mu$. Namely,
\begin{equation}
	\xymatrixcolsep{4pc}
	\xymatrix{
	0 \ar[r] & \mathcal{F} \ar[r]^-{\eta_{\lambda,\mu}}_-{\lambda i_1 \oplus \mu i_2}	& S^2(\sE') \cong \big((\bigwedge^2 V^{\vee}) \otimes \hol_{A'}(2D) \big) \bigoplus \big( (\bigwedge^2 V^{\vee}) \otimes \hol_{A'}(2D) \big)
}
\end{equation}

where $i_1, i_2$ are the respective inclusions of $\sF$ in the respective summands.

We need  to give explicitly our isomorphism: 
$$ S^2(\sE') \cong \bigwedge^2 (V^{\vee}) \otimes \hol_{A'}(2D))  \bigoplus  \bigwedge^2 (V^{\vee}) \otimes \hol_{A'}(2D)).$$

We take this isomorphism in  such a way that the respective summands correspond 
to the  following irreducible subrepresentations of the Heisenberg group inside $S^2(V^{\vee})$:
	\[  S^2(V^{\vee}) = 
	\langle y_1^2, y_2^2, y_3^2\rangle \bigoplus \langle y_2y_3, y_1y_3, y_1y_2 \rangle
	\cong  V  \bigoplus V \cong \bigwedge^2 (V^{\vee}) \bigoplus  \bigwedge^2 (V^{\vee}).
	\]

Observe in fact  that $y_i y_j =  \frac{y_1 y_2 y_3}{y_k} $, for $i,j,k$ a permutation of $\{1,2,3\}$ of signature $1$,
and similarly  $y_i \wedge y_j =   \frac{y_1\wedge  y_2 \wedge y_3}{y_k}\footnote{ here dividing by $y_k$ stands for the contraction
with the corresponding vector in the dual basis.}
 \in \CC \otimes (V^{\vee})^{\vee}$.

\bigskip

Recall that the map $\wedge s \colon 3\hol_{A'}(D)  \to \mathcal{F}\subset 3\hol_{A'}(2D)$ is given as follows
\begin{equation}
	\sigma={^t}(\sigma_1, \sigma_2, \sigma_3) \longmapsto s \wedge \sigma=
	\begin{pmatrix}
		\sigma_3 s_2 - \sigma_2 s_3 \\
		\sigma_1 s_3 - \sigma_3 s_1 \\
		\sigma_2 s_1 - \sigma_1 s_2 
	\end{pmatrix}
\end{equation}

and then

\begin{equation}
	\begin{pmatrix}
		\sigma_3 s_2 - \sigma_2 s_3 \\
		\sigma_1 s_3 - \sigma_3 s_1 \\
		\sigma_2 s_1 - \sigma_1 s_2 
	\end{pmatrix}
	\overset{\eta_{\lambda,\mu}}{\longmapsto} 
	\lambda 
		\begin{pmatrix}
		\sigma_3 s_2 - \sigma_2 s_3 \\
		\sigma_1 s_3 - \sigma_3 s_1 \\
		\sigma_2 s_1 - \sigma_1 s_2 
	\end{pmatrix}
	\bigoplus \mu 
		\begin{pmatrix}
		\sigma_3 s_2 - \sigma_2 s_3 \\
		\sigma_1 s_3 - \sigma_3 s_1 \\
		\sigma_2 s_1 - \sigma_1 s_2 
	\end{pmatrix}
\end{equation}

For each element $\s \in V^\vee \otimes \sL$ we get therefore the  following equations 

\begin{equation}
	\begin{split}
		\sigma \longmapsto \lambda 
		\begin{pmatrix}
			y_1^2(\sigma_3 s_2 - \sigma_2 s_3)  + \\
			+y_2^2(\sigma_1 s_3 - \sigma_3 s_1) + \\
			+y_3^2(\sigma_2 s_1 - \sigma_1 s_2)
		\end{pmatrix}&
		+\mu
		\begin{pmatrix}
			y_2y_3(\sigma_3 s_2 - \sigma_2 s_3) +\\
			+y_1y_3(\sigma_1 s_3 - \sigma_3 s_1) + \\
			+y_1y_2(\sigma_2 s_1 - \sigma_1 s_2)
		\end{pmatrix}.
	\end{split}
\end{equation}

The generators of the ideal sheaf $\mathcal{I}_{S'}$ of the quadruple cover $S'$ are given,
since  $\sL$ is generated by global sections,  by the images of the space   $H^0( V^{\vee} \otimes \hol_{A'}(D))$,
hence by the images of the elements
\begin{equation}
	\sigma = (\s_1,0,0), ~ (0,\s_2,0), ~ (0,0,\s_3).
\end{equation}

Namely, 
\begin{equation}
	\begin{split}
		&\s_1 F_1: = \s_1( \lambda(y_2^2s_3 - y_3^2s_2) + \mu (s_3y_1y_3-s_2y_1y_2))	\\
		&\s_2 F_2 : =\s_2  (\lambda (-s_3y_1^2+s_1y_3^2) + \mu (-s_3y_2y_3+s_1y_1y_2))\\
		&\s_3 F_3 : = \s_3 (\lambda (s_2y_1^2-s_1y_2^2) +\mu(s_2y_2y_3-s_1y_1y_3))
	\end{split}
\end{equation}

Rearranging them, and observing that $$\s_j F_j=0  \quad  \forall \s_j \in H^0(\sL) \quad  \iff \quad  F_j =0, $$

we finally obtain the following equations for $S'$:
\begin{equation}
	\begin{split}
		&F_1=s_3(\lambda y_2^2 +\mu y_1y_3) - s_2 (\lambda y_3^2 +\mu y_1y_2)	
		 =0 \\
		&F_2= s_1(\lambda y_3^2 +\mu y_1y_2) - s_3 (\lambda y_1^2 +\mu y_2y_3) = 0 \\
		&F_3= s_2(\lambda y_1^2 +\mu y_2y_3) - s_1 (\lambda y_2^2 +\mu y_1y_3) = 0.
	\end{split}
\end{equation}

We observe now that $S'$ is a subscheme of 	
 $\PP^2\times A' = \PP(V) \times A' $, which has  an action of $G$ of product type, where $G$ acts on $\PP^2=\PP(V)$ via the previous action of $\mathcal{H}_3$ on $V$, whereas  $G$ acts on $A'$ by translations.

The fact that an element $c$ in  the centre  of $\sH_3$ acts trivially on $ \PP(V) \times A' $ follows since it acts on the variables $y_i$ by 
multiplication by $c^{-1}$, since the $y_i$'s are a basis for $V^{\vee}$, while it acts on the sections $s_i$ by multiplication by $c$.
Hence the action of the Heisenberg group descends to an action of $G$.
 
We conclude that $S'$ is $G$-invariant by the following lemma.

\begin{lemma}
	The surface $S'$ is $G : = \sK(D)$-invariant.
\end{lemma}
\begin{proof}
The group $\mu_3$ acts,  if $\e$ is a primitive third root of unity,
multiplying $s_1, s_2, s_3$ by $1, \e, \e^2$, and $y_1, y_2, y_3$ by $1, \e^2, \e$: hence the equations 
$F_1, F_2, F_3$ are respectively multiplied by $1, \e^2, \e$.

The group $\ZZ/3$ acts by a cyclical permutation of $s_1, s_2, s_3$, and with the  same permutation
of $y_1, y_2, y_3$, hence $F_1, F_2, F_3$ are also cyclically permuted.

We can also show our assertions using the fact that the inclusion of $\sF$ inside $S^2(\mathcal{E}')$
was chosen to be Heisenberg equivariant.

\end{proof}

		 The above $G$-invariant equations on $\PP^2 \times A'$ 
can be written as  describing a determinantal variety of Hilbert-Burch type,  given by the vanishing of the $2\times2$ minors of the following matrix (we set $\la=1$)
		
		\[
		M=\begin{pmatrix}
			s_1 &  s_3 & s_2 \\
			y_1^2 + \mu   y_2  z_ 3 & y_3^2 + \mu   y_1 y_2 &  y_2^2  + \mu y_1y_3
		\end{pmatrix}
		\]

	\begin{remark}
		\rm {
			The global Hilbert-Burch  resolution for the ideal sheaf $\mathcal{I}_{S'}$ is the following
			\[
			0   \ra 	\hol_{A'}(-2H-2D)\oplus \hol_{A'}(-4H-D) \overset{^t M}{\ra} \hol_{A'}(-2H-D)^{\oplus 3}   \overset{(F_1,- F_3,F_2)}{\ra}  \mathcal{I}_{S'} \ra 0
			\]
	
	}
	\end{remark}

\begin{proposition}\label{smooth}
		The surface $S':=S'_{\mu}$ is smooth for a general $\mu \in \CC$ and a general pair $(A', D)$.
	\end{proposition}
\begin{proof}
It suffices to show this for $\mu=0$. Using the symmetry of these equations, and since, for general $A'$,
$s_1, s_2, s_3$ do not have common zeros, we may assume that if  $p\in S'$ is  singular, then $s_1\neq 0$
(at $p$).

Then by the criterion of bordering  minors we have two equations which locally define $S',$ namely 
$$ s_1 y_3^2 - s_3 y_1^2 = 0, \ s_1 y_2^2 - s_2 y_1^2 = 0,$$
and requiring that the respective   $z$-gradients  are proportional implies the proportionality of the vectors
$$ ( - s_3 y_1 , 0, s_1 y_3), \   ( - s_2 y_1 ,  s_1 y_2, 0),$$
and again by symmetry, since we must have $y_2 y_3=0$,  we may assume $y_3=0,$ and either $y_2=0$ or $s_3 y_1=0$.

In the first case we would have $ s_2= s_3 =0$; looking then at the gradient on $A'$, we would have  that $s_2=0$ and $s_3=0$ do not intersect transversally, a contradiction
since we know that the effective divisors $\{s_2=0\}$ and $\{s_3=0\}$ intersect transversally (see \cite[p. 776, Prop. 2.2]{pe-po4}).

If $ y_3= y_1=0$ we would have the contradiction that $s_1=0$. 

Whereas, if $y_3=s_3=0$, then  by the remark in the previous line  $ y_1 \neq 0$; if  the gradient of $s_3$ vanishes, 
then  we get a singular point of $s_3=0$. But for general $A'$ the divisors of the sections
$s_j$ are smooth, a contradiction. Hence, the gradient of the first equation on $A'$ is non zero, and we have a singular point
only if $ y_2 = s_2=0$ and the gradients of $s_3, s_2$ are proportional. But we have already seen that this is impossible.

\end{proof}

\begin{proposition}\label{F}
The family of  Heisenberg invariant deformations $\sF_t$ of the bundle $\sF$ with fixed determinant $ \det (\sF_t) = 3 D$
is smooth of dimension 2, is parametrized by $Pic^0(A')$, and consists, for  $M \in Pic^0(A')$, of the cokernel bundles
$$ \sF_M = coker ( f : \hol_{A'} (-3M) \ra V^{\vee} \otimes \hol_{A'}(D+M)), \ f = \sum_j x_j y_j,$$
for $y_1,y_2,y_3$ a canonical basis of  $V^{\vee}$ and $x_1,x_2,x_3$ a canonical basis of $H^0 (\hol_{A'} (4M + D) )$. 
\end{proposition}

\begin{proof}
The tangent space to the deformations of $\sF$ with fixed determinant is the space 
$$ H^1 ( End^0( \sF)),$$
where $ End^0( \sF)$ denotes the subbundle of trace zero endomorphisms; that is, we have
the direct sum decomposition  $ End( \sF) =  End^0( \sF) \oplus \hol_{A'} $.

The Heisenberg invariant deformations have as tangent space the subspace 
$$ H^1 ( End^0( \sF))^{\sH}.$$

By  the exact sequence

$$ (I) \ 0 \ra \hol_{A'} \ra V^{\vee} \otimes \hol_{A'}(D) \ra \sF \ra 0 $$

follows the exact sequence 

$$(II) \ 0 \ra  \sF^{\vee} \ra V \otimes \hol_{A'}(- D) \ra  \hol_{A'}  \ra 0 ,$$

hence

$$ (III) \  0 \ra V \otimes  \hol_{A'}(- D) \ra  V \otimes V^{\vee} \otimes \hol_{A'} \ra  V \otimes \sF (- D) \ra 0 ,$$

and  finally the exact sequence

$$ (IV) \ 0 \ra  \sF^{\vee} \otimes \sF = End( \sF) \ra V \otimes \sF (- D) \ra  \sF  \ra 0 .$$

The respective long exact cohomology sequences (in cases (I), (III)) yield:
$$ 0 \ra \CC \ra V \otimes V^{\vee} \ra H^0(\sF) \ra H^1 (\hol_{A'}) \ra 0, $$
$$ H^1(\sF) \cong  H^2 (\hol_{A'})\cong \CC,$$
$$ H^0 ( V \otimes \sF (- D)) \cong V \otimes V^{\vee}, \ \ $$
$$ 0 \ra V \otimes V^{\vee} \otimes H^1 (\hol_{A'}) \ra H^1 ( V \otimes \sF (- D)) \ra V  \otimes H^2 (\hol_{A'} (-D))= 
V \otimes V^{\vee}  \ra $$
$$ \ra V \otimes V^{\vee}  \otimes H^2 (\hol_{A'}) \ra H^2 ( V \otimes \sF (- D))\ra 0.$$
Taking Heisenberg invariants 
$$ H^0(\sF)^{\sH}  \cong H^1 (\hol_{A'}) , \  H^1(\sF)^{\sH}  \cong \CC , $$
where $\CC$ denotes the trivial representation, and
$$ H^0 ( V \otimes \sF (- D))^{\sH} \cong \CC, \ H^1 ( V \otimes \sF (- D))^{\sH} \cong H^1 (\hol_{A'}) \cong \CC^2, \ 
  H^2 ( V \otimes \sF (- D))^{\sH}=0.$$
 
 We take now the Heisenberg invariants of the cohomology sequence (IV), observing that by Serre duality
 $$ H^0 (End( \sF)) = H^2 (End( \sF))^{\vee}, \ H^0 (End( \sF))^{\sH} \supset \CC.$$
 We get
 $$ H^0 (End( \sF))^{\sH} \cong \CC \cong H^2 (End( \sF))^{\sH}    , \ $$
 $$  0 \ra H^1 (\hol_{A'}) \ra H^1 (End( \sF))^{\sH} \ra H^1 (\hol_{A'}) \ra 0.$$
 Since however $ End( \sF) =  End^0( \sF) \oplus \hol_{A'} $, we infer that 
 $$ H^1 ( End^0( \sF))^{\sH} \cong H^1 (\hol_{A'}), \ \ H^2 ( End^0( \sF))^{\sH}=0.$$
 
 This means that our deformations are unobstructed, with tangent space 
 $H^1 (\hol_{A'})$, which is the tangent space to $Pic^0(A')$.
 
 It is easy then to see that the universal family of deformations is our family $\{ \sF_M\}$.
 
\end{proof}

\begin{proposition}\label{inclusion of F}
Let the bundle $\sF_M$ be as in Proposition \ref{F}, and assume that 
we have a Heisenberg equivariant injective homomorphism
$$\sF_M  \to S^2(\sE')= \hol_{A'}(2 D)\otimes  S^2 ( V^{\vee} ) .$$ 

Then $2M$ is trivial, hence every deformation of $\sF$ as a (Heisenberg invariant) subbundle of  $S^2(\sE')$
is trivial. Moreover,  the homomorphism, for $M$ trivial,  belongs to the two
dimensional vector family 
described above.
\end{proposition}

\begin{proof}
By composition we obtain an equivariant homomorphism 
$$ V^{\vee} \otimes \hol_{A'}(D+M) \ra \sF_M \ra (V \bigoplus V) \otimes \hol_{A'}(2D),$$
equivalently 
$$ V^{\vee} \otimes \hol_{A'}  \ra (V \bigoplus V) \otimes \hol_{A'}(D-M),$$
determined by a homomorphism of representations 
 $$ V^{\vee}   \ra (V \bigoplus V) \otimes V.$$

To determine this we can use our previous Koszul-type arguments, observing that again $x_1, x_2, x_3$ is a regular sequence: this implies that the inclusion of $\sF_M$ factors through
$$ \sF_M \overset{\wedge s}{\hookrightarrow}  \bigg(\bigwedge^2 V^{\vee}\bigg) \otimes \hol_{A'}(2D+2M) \cong V \otimes \hol_{A'}(2D + 2M),$$
(here $s = (x_1, x_2, x_3)$ as usual) and hence our proof follows immediately, since
\begin{equation*}
	Hom ( V \otimes \hol_{A'}(2D+2M ), (V \bigoplus V) \otimes \hol_{A'}(2D))^{\sH}=
	\begin{cases}
		\begin{matrix}
			0 &  {\rm for}  \ 2M \neq 0, \\
			Hom (V, V \bigoplus V) & {\rm for}  \ 2M = 0.
		\end{matrix}
	\end{cases}
\end{equation*}
 
\end{proof}

\begin{thm}
Consider the  four dimensional family of surfaces of general type with 
$ p_g=q=2, K^2=6, d=4, \de=3,$ whose canonical models  are the quotients $X = X' / G$ of surfaces
 $X' \subset \PP^2 \times A'$ with only rational double points as singularities and defined by the following equations
\[
		M=\begin{pmatrix}
			s_1 &  s_3 & s_2 \\
			y_1^2 + \mu   y_2  z_ 3 & y_3^2 + \mu   y_1 y_2 &  y_2^2  + \mu y_1y_3
		\end{pmatrix}
		\]
		
where $\mu \in \CC$, $A'$ is a polarized Abelian surface with a polarization $\hol_{A'}(D)$ 
of type $(1,3)$, and where $\{s_1, s_2, s_3\}$ is a basis of $V = H^0(A', \hol_{A'}(D))$.

Then, this family  (which we call the PP4 family)  yields an irreducible component of the moduli space.

\end{thm}

\begin{proof}
First of all, notice that the generality condition (which holds when the Abelian surface $A'$ does not contain an elliptic curve) is an open condition.

Once this is satisfied, by the theorem of Casnati and Ekedahl $X$ is determined by $X'$ 
which in turn is determined by the Heisenberg equivariant inclusion of a subbundle $\sF$ of rank two and with $\det (\sF) = 3 D$ 
inside $ S^2(\sE') =  \hol_{A'}(2 D) \otimes  S^2 ( V^{\vee} )$. 

By Propositions \ref{F} and \ref{inclusion of F}, $\sF$ and the inclusion are determined
in an open set containing our family, and this open set is equal to our family.

\end{proof}

\begin{remark}\label{sF}
1) Proving that $\sF$ is the unique Heisenberg invariant subbundle $\sF$ of rank two and with $\det (\sF) = 3 D$
would show that  the PP4 family is the only component of the Main Stream 
of the moduli space of surfaces of general type with 
$ p_g=q=2, K^2=6, d=4, \de=3.$

2) We shall show that the closure
of our family yields a connected component of the moduli space.
 While it is clear, as in the case $K_S^2=5, d=3$, that a limit of  Tschirnaus bundles of the 
form $V^{\vee} \otimes \hol_{A'}(D)$ is again of this form, we  show now  an analogous statement for the bundle $\sF$.

\end{remark}

In the light of the previous remark, especially in the direction of part 2), we establish some characterization of 
$\sF$ and its  Heisenberg equivariant embeddings in $   S^2 ( V^{\vee})  \otimes  \hol_{A'}(2 D) $.

We assume here that $(A', D)$ is not a polarized product, hence the linear system $|D|$ has no fixed part and no base points.

Set $F : = \sF (-D)$, so that we have an exact sequence 
$$ 0 \ra \hol_{A'}(- D)  \ra  V^{\vee}  \otimes \hol_{A'} \ra F \ra 0.$$
Observe that $c_1(F) = D, \ c_2(F) = D^2 = 6$.

Using a non zero element in $V^{\vee} $, for instance the first element of the canonical basis,  we get a bundle inclusion $\hol_{A'} \ra V^{\vee}  \otimes \hol_{A'}$,
and, by composition, an exact sequence
$$ 0 \ra \hol_{A'} \ra F \ra \sI_Z (D) \ra 0,$$
where  injectivity follows since $\hol_{A'} \cap \hol_{A'}(- D)   =0$, and moreover the induced section
vanishes	 only in the 0-dimensional subscheme $$Z : = \{ s_2=s_3=0\},$$
since
$$ F / \hol_{A'} =  (V^{\vee}  \otimes \hol_{A'} )/ (\hol_{A'} \oplus \hol_{A'}(- D) ) = \hol_{A'}^2 /\hol_{A'}(- D),$$
and  the composed map $ \hol_{A'}^2 \ra F \ra \hol_{A'}( D)$ 	is given by $(s_3, - s_2)$.

Let now $\sF'$ be a vector bundle with the same Chern classes as $\sF$, and again admitting 
a Heisenberg equivariant embedding in $   S^2 ( V^{\vee})  \otimes  \hol_{A'}(2 D) $.

Set then $ F' : = \sF' (-D)$.

Assume now that $F'$ admits a non zero section, leading to an exact sequence
$$ 0 \ra \hol_{A'} (C) \ra F' \ra \sI_Z (D-C) \ra 0,$$
where $Z$ is a zero-dimensional subscheme and $C$ is an effective divisor.

Since $F'$ embeds into $   S^2 ( V^{\vee})  \otimes  \hol_{A'}( D) $, the effective divisor $C$ is	contained
in a divisor in $|D|$.

We observe that $H^0(F') $ is a representation of the Heisenberg group $\sH_3$, and if $h^0(F')\leq 2$,
the representation is a sum of 1-dimensional representations, 
hence there is a Heisenberg invariant extension, and  the  subscheme $Z$ is $\sH_3$-invariant: this implies that $9 $ divides $ |Z|$, but the length $|Z|$ equals  $6 - C (D-C) $ and since $ 6 = D^2 = C^2 + (D-C)^2 +  2  C \cdot (D-C)$, we get $ 3 \leq |Z| \leq 6$, a contradiction.

We can exclude the case $C=D$, since then $H^0(F') $  comes from the subsheaf $\hol_{A'} (C) $
(that is, $H^0(F') = H^0 (\hol_{A'} (C) )$), hence the subsheaf
is unique and the  subscheme $Z$ is $\sH_3$-invariant: this implies that $9 $ divides $ |Z|$,
again a contradiction.

We can assume therefore $h^0(F')\geq 3$, and we take the sections with a minimal curve $C$ of vanishing.

If $ 0 < C < D$, then first of all $ C \cdot (D-C) \geq 1$, since $ C + (D-C)$ is numerically connected
(this follows for instance since  $H^1 (  \hol_{A'}(- D))=0$), and moreover
$ C \cdot (D-C) \geq 2$
 since $ C \cdot (D-C) =1$ (see \cite{cat-fra}, \cite{4names}) implies that its canonical system has a base point, 
 while $|D|$ is base-point free.

Hence, from $ 6 = D^2 = C^2 + (D-C)^2 +  2  C \cdot (D-C)$ follows that either 

\medskip

1)$C \cdot (D-C) = 3$  and  both    $D-C$ and $C$ have self-intersection
zero,  or 

2) $C \cdot (D-C) = 2$ and one of them has self-intersection zero. 

In case 1) we have that one of the systems $|C|$ or $|D-C|$ has dimension $2$, hence 
$|D|$ has a fixed part, a contradiction.

In case 2), one of the systems has dimension $\leq 1$, and the other has dimension $0$.

Since $h^0(F')\geq 3$ and $Z$ is non trivial, it
follows that $dim |C| + dim |D-C| =1 $, and  $Z $ is in the base locus of $|D-C|$, which has therefore dimension $0$ and $(D-C)^2=2$.
Then, $|C|$ consists of curves which are the union of two elliptic curves $E_1, E_2$. But then $E_1 \cdot (D-C) = 1$,
hence $D-C$ maps isomorphically to the elliptic curve $A'/E_1$ and thus $D-C$ consists of two elliptic curves,
of which one is algebraically equivalent to $E_1$. This implies that  $|D|$ has a fixed part, a contradiction.

We have therefore reached the conclusion that, under the assumption of existence of a section, 
  the general section vanishes only on a finite set, and then we have
the desired exact sequence
$$ 0 \ra \hol_{A'}  \ra F' \ra \sI_Z (D) \ra 0,$$
and we know moreover  that $h^0(F') \geq 3$. 

Hence, since $|D|$ has dimension $2$ and does not have $6$ base points,
then $h^0(F') = 3$ and our exact sequence is exact on global sections.

Whence, $Z$ is the complete intersection of 2 sections of $\hol_{A'} (D)$.

Since the base-point scheme of  $H^0(F')$ is Heisenberg invariant and is contained in $Z$, 
then $F'$ is generated by global sections, and by Heisenberg invariance we have
the exact sequence 
$$ 0 \ra \hol_{A'}(- D)  \ra  V^{\vee}  \otimes \hol_{A'} \ra F' \ra 0,$$
showing that $F'$ is isomorphic to $F$.

We summarize our conclusion:

\begin{lemma}\label{F'=F}
Let $(A',D)$ be a polarization of type $(1,3)$ which is not a polarized product and let $\sF'$ be a rank 2 bundle with 
the same Chern classes as $\sF$, with a  Heisenberg equivariant embedding in $   S^2 ( V^{\vee})  \otimes  \hol_{A'}(2 D) $, and moreover with $H^0(\sF' (-D)) \neq 0$. Then $\sF'$ is isomorphic to $\sF$.

The case where  $(A',D)$ is  a polarized product cannot occur.
\end{lemma} 
 
\begin{proof}
We have given the proof assuming that $|D|$ is base-point free.

But the  only exception is when we have a polarized product of elliptic curves, namely
\begin{equation} \label{pol-prod-(1,3)}
	 (A', D) = (E_1, P_1) \times (E_2, 3 P_2). \tag{$\star \star$}
\end{equation}

In this case however we can run the same proof, in particular we infer as before that $h^0(F') \geq 3$.

We know that  the curves of the linear system $|D|$ 
consist of  a fixed elliptic curve $E'_2$ and three elliptic curves $E_1,E_2,E_3$ which are algebraically equivalent.

Hence, it is possible to have the situation $ C \cdot (D-C) = 1$: this happens 
 iff either $C=E_3$, 
or $D-C = E_3$. Moreover, we can have the situation where $C^2 = (D-C)^2 = 0$,  $C \cdot (D-C) = 3$.

If the first case occurs, since $h^0(F') \geq 3$, we must have that $Z$ is in the base locus of $|D-C|$.

Hence,  the subsheaf 
$\hol_{A'}(C) $ cannot be embedded in $   S^2 ( V^{\vee})  \otimes  \hol_{A'}( D) $.

In the second case, we get a similar contradiction if $|D-C|$ has base points.

There remains only the possibility that $C = E'_2$.

Consider now another section of $F'$: it must vanish on the curve $E'_2$, thus it gives another injective
map of $\hol_{A'}(E'_2) $ in $F'$, which yields another subsheaf, since otherwise $h^0(F') =1$.
The conclusion is that there is a map of $\hol_{A'}(E'_2)^2 \ra F'$ with nontrivial determinant,
hence $ D \geq 2 E'_2$, a contradiction.

Therefore, we have a section vanishing only on a finite number of points.

We find again a contradiction since 
$\hol_{A'} $ cannot be embedded in $   S^2 ( V^{\vee})  \otimes  \hol_{A'}( D) $
with torsion free cokernel: because all sections of 
$ \hol_{A'}( D) $ vanish on $E'_2$.

\end{proof}

\begin{thm}
The family of PP4 surfaces yields a connected component of the moduli space.
\end{thm}

\begin{proof}
The first part of our argument runs as in the case of CHPP surfaces.

We consider a 1-parameter limit $X'_0$ of PP4 surfaces, and we observe that, by virtue of  Remark \ref{sF} and Lemma \ref{F'=F},
 in the limit the bundles $\sE'$ and $\sF$ are exactly as for the PP4 surfaces; and that moreover
$(A',D)$ is not a product polarization.

The subbundle $\sF$ defines a subscheme $\Sigma \subset \PP^2 \times A'$ whose equations can be 
written   in Hilbert-Burch form, 
$$\Sigma : = \{ (y,z) \ | \ rank (M) \leq 1\}$$
\[
		M=\begin{pmatrix}
			s_1 &  s_3 & s_2 \\
			y_1^2 + \mu   y_2  y_ 3 & y_3^2 + \mu   y_1 y_2 &  y_2^2  + \mu y_1y_3
		\end{pmatrix}
\]
		
		and it would suffice  to see that $\Sigma$ is normal: this  however is not always the case,
		as we shall see later on. Therefore we use another argument.
				
Since $|D|$ has  no base points, $\Sigma$ is the union of the open sets $\sU_j \cap \Sigma,$ where $ 
\  \sU_j : =  \{ s_j \neq 0\}$.

By symmetry, we may analyze the open set   $\sU_1$ where $s_1 \neq 0$, and set 
		$$s: = \frac {s_2}{s_1}, \ s': = \frac {s_3}{s_1},  \ q_j : = y_j^2 + \mu   y_{j+1}  y_ {j+2}$$
(here the indices have to be understood as elements of  $\ZZ/3$).
	
Then, on this open set we have, by the criterion of bordering minors, the complete intersection	
$$ \Sigma \cap  \sU_1 = \{ q_2 - s q_1 = q_3 - s' q_1=0 \}.$$
Hence, $\Sigma$ is Gorenstein, with dualizing (canonical) sheaf  $\omega_{\Sigma} = \hol_{\PP({\sE'}^{\vee})}(1)$,
and thus the vector space of sections of the  canonical system is generated  by the pull-back of the 2-form on $A'$
and by the elements $\{y_i s_j\}, i,j \in \{1,2,3\},$ which are a basis of $V^{\vee} \otimes H^0(  \hol_{A'}( D))$.

Observe that the sub-system generated by $\{y_i s_j\}, i,j \in \{1,2,3\}$, is base point free and factors 
through $ \Sigma \ra \PP^2 \times A' \ra \PP^2 \times \PP^2 $, where the last map is the product of the identity
with $\varphi_D$.

Assume now that $\Sigma$ is not normal, and let $X''$ be a resolution of singularities of $\Sigma$: then 
$ p_g(X'') < 10$, since the sections of $\hol_{X''} (K_{X''})$ correspond to a subspace of
the canonical system of $\Sigma$ contained in the   subspace  of sections vanishing on the singular curve
of $\Sigma$.

This is a contradiction, since $X'_0$ is birational to $X''$ and  $p_g(X'_0) = 10$ (since $\chi( X'_0)=9, q( X'_0)=2$).

A similar contradiction is found if $\Sigma$ is normal but its singularities are not rational double points.

\end{proof}

In order to give a more explicit description of the surfaces in  this connected component of the moduli space, it is 
desirable  to describe 
the singular sets of such surfaces $\Sigma$.

\bigskip

\section{Singular sets of extended PP4 surfaces $\Sigma$}

As in the previous Theorem, we consider an extended PP4 surface 

$$\Sigma : = \{ (y,z) \ | \ rank (M) \leq 1\} \subset \PP^2 \times A'$$

\[
		M=\begin{pmatrix}
			s_1 &  s_2 & s_3 \\
			q_1 &  q_2 & q_3
		\end{pmatrix}.
		\]
		
To avoid cumbersome calculations, we write $$q_j (y) : = y_j^2 + 2m y_{j+1}y_{j-1},$$
observing right away that $ 3 q_j (y)  =\frac{ \partial f_m}  {\partial y_j}$,
where 
$$ f_m (y) = \sum_j y_j^3 + 6 m y_1 y_2 y_3.$$

Hence, $ R_m : = \{ y \ | \ f_m (y) = 0\}$ is a cubic in the Hesse pencil of cubic curves,
which is a smooth cubic or the union of 3 lines: the latter situation occurs precisely 
for $\mu^3 = (2m)^3 = -1$.

The crucial remark is that $\Sigma$ is a birational fibre product
$$ \Sigma = \{ ( y,z) \ | \ q(y) = \varphi_D(z)\} = \PP^2 \times_{ \PP^2 } A',$$
$$\varphi_D(z) = s(z): = (s_1(z),s_2(z),s_3(z)), \ \ q(y) : = (q_1(y),q_2(y), q_3(y)). $$

Observe that $s = \varphi_D$ is always a finite morphism, since we assume that $(A',D)$
is not a polarized product, whereas $q$ is a morphism (hence also finite) if $(2m)^3 \neq -1$.

If instead $(2m)^3 =  -1$, $R_m$ consists of three lines, hence $q$ is a standard birational 
Cremona transformation contracting the three lines and blowing up the three singular points
of $R_m$. 

From this remark and a trivial calculation in local coordinates follows that, defining 
$R_m$ to be the ramification divisor for $q$, $R$ as the ramification divisor for $s$,
the point $(y,z)$ is a smooth point if either $ y \notin R_m$ or $ z \notin R$.

Similarly, $(y,z)$ is a smooth point if  $ y \in R_m,  z \in R$,
but the rank of the derivatives are $rk (Dq_y) =1,$ $rk (Ds_z) =1,$
and the respective  images of $Dq_y, Ds_z$ are not the same tangent line.

A partial conclusion is that, defining 
$B_m$ to be the branch  divisor for $q$, $\sB $ the branch  divisor for $s$,
the singular points lie above the points in the plane in $B_m \cap \sB$ where 
the two curves do not intersect transversally.

It follows then:

\begin{proposition}
An extended PP4 surface $\Sigma $ is normal unless the two curves $\sB$ and $B_m$
have a common component or $(2m)^3 = -1$.
\end{proposition}
\begin{proof}
We have seen that $\Sigma $ is a local complete intersection, whence it is normal
if and only if it is smooth in codimension one; that is, $Sing(\Sigma) $ is a finite set.

If $\sB$ and $B_m$
have no common component, their intersection is a finite set. We use then the fact 
that $s = \varphi_D$ is a finite morphism, and also $q$ is a morphism for  $(2m)^3 \neq -1$,
hence necessarily finite.

\end{proof}

The reason for assuming $(2m)^3 \neq  -1$ is that, 
for $(2m)^3=  -1$, $R_m$ consists of a triangle, and if $y'$ is a vertex of the triangle,
then $q(y')=0$. Thus, it follows that $\Sigma \supset \{y'\} \times A'$,
hence $\Sigma$ has at least  four irreducible components.

For  $(2m)^3 \neq  -1$ the branch locus $B_m$ consists of the dual sextic curve to the cubic
$R_m$, which has equation (see \cite{casnati99}, page 383)
\begin{equation*}
	\begin{split}
		B_m : = \{ x \ \ | \ \ &\sum_j x_j^6 + 2 ( - 16 m^3 - 1) (\sum_{i\neq j} x_i^3  x_j^3 ) \\
		&- 24 m^2  x_1 x_2 x_3 (\sum_j x_j^3 ) + 6m (-8m^3 -4) x_1^2 x_2^2 x_3^2=0\}.
	\end{split}
\end{equation*}

Using \cite{casnati99} we have:

\begin{example}
There exist extended PP4 surfaces which are not normal.

It suffices to take $(A',D)$ a bielliptic Abelian surface of type $(1,3)$ with $B_m$ contained in the branch locus $\sB$.
These exist by \cite{casnati99}.

\end{example}

\bigskip

\section{The PP4 family and the construction of \cite{pe-po4}}
	
	We shall now show  that the 4-dimensional family of PP4 surfaces
	contains  the family of  surfaces described in \cite{pe-po4}.

		Recalling that $\al' : S' \ra A'$ is  in general a finite quadruple cover,
		we have\footnote{In \cite{h-m} $(\sE')^\vee$ is called $\sE'$.}
		$$\al'_\ast \hol_{S'}=\hol_{A'} \oplus (\sE')^\vee= \hol_{A'} \oplus (V \otimes \hol_{A'}(-D)). $$
		The multiplication tensor is given by two tensors 
		$$\tau_0 :  (\sE')^\vee \times (\sE')^\vee \ra  \hol_{A'}, $$
		$$\tau_1 : (\sE')^\vee \times (\sE')^\vee \ra (\sE')^\vee. $$
		As Hahn and Miranda  prove \cite{h-m}, $\tau_0$ is determined by $\tau_1\in H^0(A', S^2(\sE') \otimes (\sE')^\vee)=H^0(A', S^2(V^\vee) \otimes  V \otimes \hol_{A'}(D))= S^2(V^\vee ) \otimes V \otimes V$. 
		
		Indeed, they show that, since $\bigwedge^2\sF\cong \bigwedge^3  \sE' $, $\tau_1$ is in turn determined by a totally decomposable section  (that is, a section which locally is the wedge product of two 
		local sections):
		$$\ga_{\la,\mu} \in H^0(A', \bigwedge^3 (\sE')^\vee \otimes \bigwedge^2 S^2(\sE'))=H^0(A', \bigwedge^3 V \otimes \bigwedge^2 S(V^\vee ) \otimes \hol_{A'}(D))=\bigwedge^3 V  \otimes \bigwedge^2 S^2(V^\vee) \otimes V .$$
		An easy calculation yields
		\[
		\begin{split}
			\ga_{\la,\mu}=(s_1 \wedge s_2 \wedge s_3) \otimes \big
			(& -\la \mu s_2 \otimes y_1^2 \wedge y_1y_2 +  \la \mu s_3 \otimes  y_1^2 \wedge y_1y_3 + \la^2 s_3 \otimes y_1^2 \wedge y_2^2	\\ 
			& + 0 \otimes y_1^2 \wedge y_2y_3   -\la^2 s_2 \otimes y_1^2 \wedge y_3^2 -\mu^2 s_1 \otimes y_1y_2 \wedge y_1y_3 \\
			& - \la \mu s_1 \otimes y_1y_2 \wedge y_2^2 + \mu^2 s_2 \otimes y_1y_2 \wedge y_2 y_3 + 0 \otimes y_1 y_2 \wedge y_3^2 \\
			&+ 0 \otimes y_1 y_3 \wedge y_2^2 - \mu^2 s_3 \otimes y_1 y_3 \wedge y_2 y_3 + \la \mu s_1 \otimes y_1y_3 \wedge y_3^2 \\
			& -\la \mu s_3 \otimes y_2^2 \wedge y_2y_3 + \la^2 s_1 \otimes y_2^2 \wedge y_3^2 -\la \mu s_2 \otimes y_2y_3 \wedge y_3^2\big)
		\end{split}
		\]
		As pointed out in \cite[p. 12]{h-m}, $\ga_{\la, \mu}$ corresponds to the Plucker embedding
		$$\bigwedge^2 \sF \ra \bigwedge^2 S^2(\sE').$$
		Choosing for $S^2(V^\vee)$ the ordered basis 
		\[
		\begin{split}
			\{ &   y_1^2 \wedge y_1y_2, \ y_1^2 \wedge y_1y_3, \ y_1^2 \wedge y_2^2, \ y_1^2 \wedge y_2y_3, \ y_1^2 \wedge y_3^2, \\
			&  y_1y_2 \wedge y_1y_3, \ y_1y_2 \wedge y_2^2, \ y_1y_2 \wedge y_2 y_3, \ y_1 y_2 \wedge y_3^2, \ y_1 y_3 \wedge y_2^2, \\
			& y_1 y_3 \wedge y_2 y_3, \ y_1y_3 \wedge y_3^2, \ y_2^2 \wedge y_2y_3, \ y_2^2 \wedge y_3^2, \ y_2y_3 \wedge y_3^2\}
		\end{split}
		\]
		as in \cite{pe-po4}, under the identification 
		$$X,Y,Z\longleftrightarrow s_1,s_2,s_3, \qquad \hat{X}, \hat{Y}, \hat{Z} \longleftrightarrow y_1,y_2,y_3,$$
		we can write
		\[
		\begin{split}
			\ga_{\la, \mu} = 
			\big( & -\la \mu s_2,\  \la \mu s_3,\  \la^2 s_3, \ 0, \   -\la^2 s_2 ,\\
			&-\mu^2 s_1, \ - \la \mu s_1,  \mu^2 s_2, \ 0, \ 0, \\
			&- \mu^2 s_3, \ \la \mu s_1, \ -\la \mu s_3, \ \la^2 s_1, \ -\la \mu s_2 \big) \in H^0(A', \hol_{A'}(D))^{\oplus 15}.
		\end{split}
		\]
		Then, it is easy to see that  $\gamma_{\la, \mu}$ has the same form provided by Penegini and Polizzi in \cite{pe-po4} Proposition 2.3, setting
		\[
		\begin{matrix}
			a:=-\la \mu, & b:= \la \mu, & c:=\la^2, \\
			d:=0, & 	e:= -\mu^2,
		\end{matrix}
		\] 
		and it fulfills the properties stated in Proposition 2.4.
		
		Therefore, this shows that the family of PP4 surfaces contains the one in \cite{pe-po4}.
		
		\begin{remark}
			Note that in \cite{pe-po4}, in the statement of Proposition 2.3, we have to switch $Y, Z$ since in our case the dual of the Heisenberg representation $V^\vee$ is equivalent to the one given in \cite{pe-po4} via the matrix
			$$A:=
			\begin{pmatrix}
				1 & 0 & 0\\
				0 & 0 & 1\\
				0 & 1 & 0
			\end{pmatrix}. $$
			Furthermore, we point out that the correspondence between the parameters $(\la, \mu)$ and $(a,c)$    giving respectively the family of PP4 surfaces and   the family in \cite{pe-po4}  is as follows
			$$ (\la, \mu) \longmapsto (-\la \mu, \la^2).$$
		\end{remark}

	\subsection{Branch Locus}
	We  compute now  the branch locus of the degree 4 cover
	$$ \alpha'\colon   S' \ra A',$$
	induced by   the second projection map  $\PP(V)\times A' \ra A' $. 
	
	The equations of $S'$ are given by the vanishing of the $2\times 2$ minors of the matrix $M$, which we have seen to equal 	
	\begin{equation} \label{explict_eq_S'}
		\begin{cases}
			F_1=s_3(  y_2^2 + \mu   y_1y_3)- s_2 ( y_3^2  + \mu  y_1y_2     )=0\\
			F_2= s_1(  y_3^2 + \mu   y_1y_2) - s_3 ( y_1^2  + \mu  y_2y_3  )=0 \\
			F_3= s_2 (  y_1^2 + \mu   y_2y_3)- s_1 ( y_2^2  + \mu  y_1y_3   )  =0			
		\end{cases}
		\quad \mu \in \CC
	\end{equation}
	 
	A fixed fibre is singular if and only if it has strictly less than $4$ points. 
	
	We are going to  find the conditions which the coefficients  of the system (\ref{explict_eq_S'}), namely $s_1,s_2,s_3, \mu$, must satisfy in order that this happens.  
	
	Recall that the case
	\[
	s_1=s_2=s_3=0
	\]
	cannot occur.
	
	Then, one of them is nonzero and by symmetry   we may assume that
	\[
	s_1\neq 0.
	\]
	Then, as in Proposition \ref{smooth}, the local equations are 
	
	$$ 
		F_2=0, \ \ 
		F_3=0 .$$
	
	The equations  describe the intersection of two conics in $\PP^2$ and the intersection points are
	exactly the base points of the pencil generated by them.

	Let us denote by  $ A_{s,t}\in Mat(\CC, 3)$ the  $3\times 3$ symmetric matrix of a conic in the pencil, namely
	\begin{equation}
		A_{s,t}=
		\begin{bmatrix}
			s_2t -s_3s & \frac{1}{2}\mu s_1 s & -\frac{1}{2}\mu s_1t\\
			\frac{1}{2}\mu s_1s & -s_1t  &  \frac{1}{2}\mu( s_2t -  s_3s)\\
			-\frac{1}{2}\mu s_1t &  \frac{1}{2}\mu(s_2t - s_3s) & s_1s 
		\end{bmatrix},
	\end{equation}
	and consider its determinant
	$$p(s,t):=\det A_{s,t}\in \CC[s,t]_3.$$
	We point out that the the case $p(s,t)\equiv 0$ does not occur for a general fibre since $S'$ is 
	irreducible, hence
	 the pencil does not have any fixed component.

	 $p(s,t)$ is then a nonzero homogeneous polynomial of degree $3$ whose  roots correspond to the degenerate conics of the pencil. Recall however that the base points of a pencil of conics are less then $4$ if and only if the pencil contains at most 2 degenerate conics, that is $p(s,t)$ has at least one multiple root, equivalently
the  discriminant of $p(s,t)$ vanishes.

	Since we have that
	\begin{equation*}
		\begin{split}
			p(s,t)&= \frac{1}{4}\mu^2(s_3^3-s_1^3)s^3 +\\ 
			&+\bigg(\frac{1}{4}\mu^3 s_1^2s_3 - \frac{3}{4} \mu^2 s_2s_3^2 + s_1^2s_3 \bigg)s^2t+\\
			&+\bigg(-\frac{1}{4}\mu^3 s_1^2s_2 + \frac{3}{4} \mu^2 s_2^2s_3 - s_1^2s_2 \bigg)st^2
			+\\
			&+\frac{1}{4}\mu^2(s_1^3-s_2^3)t^3,
		\end{split}	
	\end{equation*}
	
	and we have the well-known formula saying that  the discriminant of a polynomial 
	$ p = a x^3 + b x^2 + c x + d $ equals
	
	 $$ b^2 c^2 - 4 a c^3 - 4 b^3 d - 27 a^2 d^2 + 18 abcd,$$

	a long but straightforward computation shows that the equation of the
	discriminant  is in our case:
	\begin{equation} 
	\begin{split}
			&s_1^6\bigg[-27\mu^8(s_1^6+s_2^6+s_3^6)+\mu^2\bigg(-4\mu^9 + 6\mu^6 - 192\mu^3
			- 256\bigg)(s_1^3s_2^3+s_1s_3+s_2^3s_3^3)+ \\
			&+\mu^{4}\bigg(18\mu^{6} + 144 \mu^3+ 288\bigg)s_1s_2s_3(s_1^3+s_2^3+s_3^3)+\\
			&+\bigg(\mu^{12} - 92\mu^{9} - 336\mu^{6} + 256\mu^{3} + 256\bigg)s_1^2s_2^2s_3^2\bigg] =0
		\end{split}
	\end{equation}
	
	Since we worked on the open set  $s_1\neq 0$, we finally get the branch locus equation 
	
	\begin{equation} \label{branch_locus_equation}
		\begin{split}
			&-27\mu^8(s_1^6+s_2^6+s_3^6)+\mu^2\bigg(-4\mu^9 + 6\mu^6 - 192\mu^3
			- 256\bigg)(s_1^3s_2^3+s_1s_3+s_2^3s_3^3)+ \\
			&+\mu^{4}\bigg(18\mu^{6} + 144 \mu^3+ 288\bigg)s_1s_2s_3(s_1^3+s_2^3+s_3^3)+\\
			&+\bigg(\mu^{12} - 92\mu^{9} - 336\mu^{6} + 256\mu^{3} + 256\bigg)s_1^2s_2^2s_3^2 =0.
		\end{split}
	\end{equation}

	One easily sees by symmetry that the  cases $s_2\neq 0$, $s_3\neq 0$ lead to exactly the same equation. 
	
	\begin{remark}
		\rm {
			Note that this is exactly the same branch locus as the one found by Penegini and Polizzi in \cite[p. 749, equation (14)]{pe-po4}. In fact,  by multiplying  their equation with $-27$ and setting $c=1$, $a=-\mu$, $X=s_1$, $Y=s_2$, $Z=s_3$, one gets (\ref{branch_locus_equation}).
			
		}
	\end{remark}

\section{Analysis of the case of degree $d=3$ 	under the generality  assumption}
	
We have that $\al : S \ra A$ has degree $d=3$, hence $\sE, \sE'$ have rank equal to $2$ 
and moreover there is a Heisenberg-equivariant exact sequence \eqref{MainSeqGenAss}, namely
\begin{equation*}
	0 \ra \mathfrak H' \ra \sL \otimes V^\vee   \ra  \sE'\ra 0.
\end{equation*}

Since $\sL \otimes V^\vee$ has rank $\de$, $\mathfrak H'$ has rank $\de - 2$, and by definition
it is a succesive extension of line bundles in $Pic^0(A')$.

The main result of \cite{c-e} works without modification under the Gorenstein
Assumption, and it follows that $S'$ is the divisor in $\PP(V) \times A'$ corresponding to 
a section of $ Sym^3(\sE') \otimes (\det(\sE'))^{-1}$.

The total Chern class of $\sE' $ equals  $$ c(\sE') = ( 1 + D)^{\de},$$
and, letting $H$ be the hyperplane divisor of $\PP: = Proj(\sE')$, by the Theorem of Leray-Hirsch
we get
$$ H^2 - c_1(\sE') H + c_2(\sE') =0.$$
 We observe  that the class of $S'$ equals
 $ 3 H - c_1(\sE')$, and $K_{\PP} = - 2 H + c_1(\sE')$, hence $K_{S'} = H | _{S'}$.
 
 Hence $$K_{S'}^2  = H^2 (3 H - c_1(\sE'))= (c_1(\sE') H - c_2(\sE')) (3 H - c_1(\sE')) = $$
 $$ = 3 c_1(\sE') H^2 - c_1(\sE')^2 - 3 c_2(\sE') = 2 c_1(\sE')^2 - 3 c_2(\sE') = 2 \de^2 D^2 - \frac{3}{2} \de (\de-1) D^2=$$
 $$ = \de^2 ( 4 \de - 3 (\de - 1)) = \de^2 (\de + 3).$$
 Since the degree of $S' \ra S$ equals $|G| = \de^2$, we have shown the first assertion of the following
 \begin{proposition}\label{K^2}
 $K^2_S = \de+3$, and $\chi(S) = 1$.
 \end{proposition}

 \begin{proof}
 There remains to show the second assertion.
 
 It suffices to  show that  $\chi(S') = \de^2$. This follows from the exact sequence 
 $$ 0 \ra  \hol_{\PP}  ( K_{\PP} ) = \hol_{\PP} ( - 2 H + c_1(\sE') ) \ra \hol_{\PP} (H) \ra \hol_{S'} ( K_{S'} ) \ra 0.$$
 
 We have  $\chi (\hol_{\PP}  ( K_{\PP} ) ) = -  \chi (\hol_{\PP} )= 0$,
 more precisely, $h^0 (\hol_{\PP}  ( K_{\PP} )) =0$,  $h^1 (\hol_{\PP}  ( K_{\PP} ) ) = h^2 (\hol_{\PP} ) = 1$,  $h^2 (\hol_{\PP}  ( K_{\PP} )) =h^1 (\hol_{\PP}) = 2$, $h^3 (\hol_{\PP}  ( K_{\PP} )) =1$.
 
 On the other hand, $h^i(\hol_{\PP} (H) ) = h^i( \sE')$ yields    
 $$ \chi (S')  = \chi (\hol_{S'} ( K_{S'} ))=  \chi (\sE') = \de \chi (\sL )  - \chi (\mathfrak H') = \de^2- 0 = \de^2,$$
  as we wanted to show.

 \end{proof}
 
 \begin{remark}\label{q=2}
 In the above proposition we certainly have $q(S')=2$ provided that $ h^1( \sE')=0, h^2( \sE')=0$,
 and this follows if $h^i(\mathfrak H')=0$ $\forall i$.
 
 \end{remark}

  The case $\de=2$ is the case of CHPP surfaces, that we have already described in detail, so let us proceed
  to the next case $\de=3$.
  
Here, in the  Heisenberg-equivariant exact sequence
\begin{equation} \label{MainSeqGenAss-de=d=3}
	0 \ra \mathfrak H' \ra \sL \otimes V^\vee   \ra  \sE' \ra 0 \tag{$\bullet \bullet$}
\end{equation}

$ \mathfrak H'$ is a line bundle in $Pic^0(A')$.

\subsection{The case where the homogeneous bundle is trivial} \label{triv-hom-bdl}

The first guess here is to take $ \mathfrak H' = \hol_{A'}$, so that 
the inclusion $j : \hol_{A'}  \ra \sL \otimes V^\vee$ is given by a Heisenberg invariant
section of $H^0 ( \sL \otimes V^\vee) = V  \otimes V^\vee$.

Since $V$ is an irreducible representation, the only invariant by Schur's lemma corresponds to
the identity of $V$, hence to the element $ \sum_j x_j y_j$, where $x_1,x_2, x_3$ is the natural basis
of $V$ and  $y_1,y_2, y_3$ is the dual basis.

In order to get a section of $ Sym^3(\sE') \otimes (\det(\sE'))^{-1}$ we use the surjection
$$ Sym^3 ( V^\vee) \otimes \hol_{A'} \ra  Sym^3(\sE') \otimes (\det(\sE'))^{-1},$$
hence the surjection $ Sym^3 ( V^\vee)  \ra  H^0 (Sym^3(\sE') \otimes (\det(\sE'))^{-1}),$
and we consider a cubic form $F(y) \in Sym^3 ( V^\vee)$.

Then, $F(y) = 0$ defines a cubic curve $C \subset \PP^2 = \PP( V)$, and 
$$ S' : = \{ (y,z) \in \PP( V) \times A' \ | \ \sum_j y_j x_j(z) = 0 , \ F(y)=0\}.$$

  $S'$ is an ample divisor inside the Abelian variety $Z: = C \times A'$, hence, by Lefschetz theorem, $ q(S')=3$;
  moreover the class of $S'$ is $H + D$, so $K_{S'}^2  =   (H+D)^3 = 3 H D^2 = 3 \cdot 2 \de^2 = 6 \cdot 9$,
  hence this calculation shows once more that $S$ has $K^2_S= 6$, while the exact sequence
  $$ 0 \ra  \hol_{Z}    \ra \hol_{Z} (H + D ) \ra \hol_{S'} ( K_{S'} ) \ra 0$$
  shows that $\chi(S') = \chi (\hol_{Z} (H + D )) = 9$, hence $\chi(S)=1$.
  
  We have taken for granted that $S'$ is   $G$-invariant and smooth, let's now show it.
  
  First of all, $C$ must be $G$-invariant, and since $G$ is generated by a cyclical permutation
  $ g_1 : y_1 \mapsto y_3 \mapsto y_2 \mapsto y_1$, follows that $F$ is a linear combination of
  $$\sum_j y_j^3, \quad  y_1y_2y_3, \quad  \sum_i y_i^2 y_{i+1}, \quad  \sum_i y_i^2 y_{i-1}.$$
  Since the other generator $g_2$ of $G$ acts via the diagonal matrix with entries $ 1, \e^2, \e$,
  the above monomials are eigenvectors for respective eigenvalues $ 1, 1, \e^2, \e$,
  so that either
  \begin{enumerate}
  	\item $ F =  \sum y_j^3 +\la  y_1y_2y_3$, or
  	
  	\item $F=  \sum_i y_i^2 y_{i+1}$, or
  	
  	\item $F=  \sum_i y_i^2 y_{i-1}$.
  \end{enumerate}

   The third case is projectively equivalent to the second, via the projectivity 
   $\iota_2$ which exchanges the coordinates $y_2, y_3$, but the isomorphism does not preserve the action of $G$. 
   
   However, see \cite{la-se}, we can use the involution $\iota$ coming from the extended Heisenberg
   group, and such that $ \iota (x_1, x_2, x_3) = (x_1, x_3, x_2) $: it is associated to an automorphism $\iota$ of $A'$
   which equals to multiplication by $-1$ for a suitable choice of the origin.
   
   Now, the isomorphism $ \iota \times \iota_2$ normalizes the action of the group $G$
   (sending each element of $G$ to its inverse), and leaves the equation
   $ \sum_i x_i y_i=0$ invariant.
   
   Hence, in the sequel we shall restrict ourselves to consider only the first and second case.

   In the first case $C$ is a curve of the Hesse pencil of cubics, hence it is either smooth or the product of three linear forms,
    in the second case $C$ is smooth, since the three polynomials 
 	\[
 	y_1^2 + 2y_2y_3, \quad y_2^2 + 2y_1y_3, \quad y_3^2 + 2y_1y_2
 	\]
 	cannot vanish simultaneously (one observes that $y_j \neq 0$ for all $j$, then $y_1=1$
	implies $|y_2| = |y_3|=2$, hence $1 = 8$, a contradiction).  
    
	\bigskip 
	
	\bigskip
	
    \subsubsection{A family of surfaces with $p_g=q=3, \ K^2=6, \ d=3$} \label{Hesse-cubic}
    
    In the first case, we have a linear system on $\PP({\sE'}^{\vee})$, which has as base-point set the  intersection
    of $\PP({\sE'}^{\vee})$ with $\sB \times A'$, where $\sB$ is the base-point set of the Hesse pencil,
    consisting of the $9$-point orbit via the symmetric group $\mathfrak S_3$
    of the points  $(0,1,-\e^j)$, where $\e$ is  a primitive cubic root of $1$.
    
 By the first Bertini's theorem it suffices to show smoothness at these points, and by cyclic symmetry it suffices to
 look at the points  $(0,1,-\e^j) : = (0,1,\zeta)$.
 
 Notice that we have a smooth point of $S'$ if the gradients of the two equations
 are not proportional in $(y,z)$. This certainly happens if $z$ is a smooth point of the curve $D_y : = \{ z \ | \ \sum_j y_j x_j(z)=0\}$.

 Now, for a general $A'$, the curve $D_{y}$ is always irreducible, hence it has only a finite number of singular points.
 Making $y$ vary in the finite set of the $9$ base points of the Hesse pencil, 
 we only get a finite set of points in the plane $\PP^2$ with coordinates $(x_1, x_2, x_3)$,
 image of $A'$ under the finite morphism $\varphi_D$ of degree $6$ associated to $H^0(A', \hol_{A'}(D))$. 
 
 These points must then satisfy the equation $ x_2 + \zeta x_3=0$, coming from the equation $\sum_j x_j y_j=0$.
 Moreover, the condition that the gradients are proportional means that the rank of the following matrix 
 \[
 N=\begin{pmatrix}
 	 \la \zeta&  3 & 3 \zeta^2 \\
 	x_1  &  x_2  &  x_3
 \end{pmatrix}
 \]
 is at most $1$.
 This leads to further equations
 $$ 3 x_1 - \la \zeta x_2 =  x_3 - \zeta^2 x_2 = - \la x_3 + 3 \zeta x_1 = 0.$$
 We may set  $x_1 = 1$, hence we get 
 $$ x_2 = \frac{3}{\la \zeta} , \quad  x_3 =  \frac{3 \zeta }{\la}.$$

	Therefore,  the point $(x_1,x_2,x_3):=(\la \zeta^2, 3 \zeta, -3)$ is  the only   solution in the plane.
	
	Then, $z$ must be a singular point for the curve $D_{(0,1,\zeta)}$, and belong to the preimage
	  of $(\la \zeta^2, 3 \zeta, -3)\in \PP^2$ via the degree $6$ morphism
	\[
	\varphi_{D} \colon A' \ra \PP^2
	\]
	associated to $H^0(A', \hol_{A'}(D))$. 
		
	This cannot happen for a general pair $(A',D)$ since then the  divisor 
	
	\[
	D': \quad \{ x_2(z) + \zeta x_3(z)=0\}
	\]
	
	 is  smooth.

     The action of $G$ on a curve of the Hesse pencil is free, since $g_1$ has as fixed points only the three points
     $(1, \e^j, \e^{2j})$, $j=0,1,2$,  while $g_2$ 
     has as fixed points only the coordinate  points, and all these points do not belong to the general cubic $C$.
     
     The conclusion is that $G$ acts by translation via the nine  $3$-torsion points of $C$,
     hence  $q(S)=3$.
    
      \bigskip

     \subsubsection{A new component, consisting of surfaces with $p_g=q=2$, $K_S^2=6$, $d=3$.} \label{new-family}

     The main point to establish, in the case of the curve
     $$ C : = \{ f(y) : = \sum_i y_i^2 y_{i+1} =0\},$$
     is that the surface $S'$ is again smooth (or has only rational double points as singularities).

     This is done in Theorem 0.2 of \cite{ca-se}.

      In   this case  the action is not  free on $C$, since the	 coordinate points belong to $C$ 
      and are fixed for $g_2$. Thus, the quotient surface $S = S' /G$ has $q(S)=2$ and $\chi(S)=1$,
      as desired.
      
      Let us therefore discuss the singularities of $S'$, which has equations

     $$ S' : = \{ (y,z) \in \PP( V) \times A' \ | \ \sum_j y_j x_j(z) = 0 , \sum_i y_i^2 y_{i+1} =0\}.$$
 	
 	We have	 already shown that $C$ is smooth.

 	$(y,z)$ is  a singular point of $S'$ if and only if  $z$ is  a singular point of the curve $D_y:=\{z \ |  \ \sum_j y_j x_j(z)=0\}$ and the rows of the matrix
 	\begin{equation}
 		\begin{pmatrix}
 			y_3^2 + 2y_1y_2 & y_1^2 + 2y_2y_3 &  y_2^2 + 2y_1y_3  \\
 			x_1 & x_2& x_3
 		\end{pmatrix}
 	\end{equation}
	are  proportional. This means  that
	\begin{equation}
		x: = (x_1, x_2,x_3)= \nabla f (y), \quad  y: = (y_1, y_2, y_3),
	\end{equation}
	
	and 	we view $x$ as a point of $(\PP^2)^{\vee}= : \PP'$, while $ y \in \PP : = \PP^2$.
	
Geometrically, this means that $ x \in C^{\vee}$, and $x$ represents a tangent line to $C$ at $y$,
hence $y$ represents a line $\Lam_y$ tangent to $C^{\vee}$ at $x$.

Moreover, since $z$ is a singular point of $D_y$, which is the inverse image under $\varphi_D$
of the line $\Lam_y$ corresponding to $y$, we require  that the	 line $\Lam_y$ is tangent at $x$
to the branch curve $\sB$ of $\varphi_D$. Hence, that $\sB$ and $C^{\vee}$ are tangent.

Therefore, $S'$ is smooth if $\sB$ and $C^{\vee}$ intersect transversally.

The following is the content of Theorem 0.2 of \cite{ca-se}:
\begin{thm}\label{branchconj}
Let $\sB$ be the branch curve of $\varphi_D : A' \ra \PP^2$, where $D$ is a polarization of type $(1,3)$
and the pair $(A',D)$ is general.

Then, $C$ being the plane curve $ C : = \{  \sum_i y_i^2 y_{i+1} =0\}$, 
$\sB$   intersects transversally the dual sextic curve $C^{\vee}$ and $C$ intersects transverssally the discriminant
curve $W$ of the linear system $|D|$.

\end{thm}

     \bigskip

    \subsection{The case where the homogeneous bundle is nontrivial} \label{nontriv-hom-bdl}
    
    The next option  is to take  $ \mathfrak H' =\sM=\hol_{A'}(M)$, a nontrivial line bundle in $Pic^0(A')$.
    Observe  that the inclusion $ \sM \hookrightarrow \sL \otimes V^{\vee}$ comes from a section  
    $$\xi \in H^0(   \sL(-M)  \otimes V^{\vee}),$$
    and there is a point $ z \in A'$ such that, if $t_z$ denotes the translation by $z$, 
    $$\xi \in t_z^* H^0(   \sL ) \otimes V^{\vee}.$$

    Then, from the exact sequence \eqref{MainSeqGenAss-de=d=3},  which in this case reads out as
    \begin{equation*}
    	0 \ra \hol_{A'}(M) \ra \sL \otimes V^\vee   \ra  \sE' \ra 0,
    \end{equation*}

     we get the following  Eagon-Northcott exact sequence for $Sym^3(\sE')$:
    
    \begin{equation} \label{eag-north}
    	\xymatrixcolsep{1pc}
    	\xymatrix{	
    		0 \ar[r] & Sym^2(V^\vee)\otimes \hol_{A'}(2D+M) \ar[r]  &  Sym^3(V^\vee)\otimes \hol_{A'}(3D)  \ar[r]& Sym^3(\sE') \ar[r] & 0.
    	} \tag{$+$}
    \end{equation}
    
    By tensoring with $\det(\sE')^{-1}=\hol_{A'}(M-3D)$ we get:

    \begin{equation} \label{twisted-eag-north}
    	\xymatrixcolsep{1pc}
    	\xymatrix{	
    0 \ar[r] & Sym^2(V^\vee)\otimes \hol_{A'}(-D+2M) \ar[r]  &  Sym^3(V^\vee)\otimes \hol_{A'}(M)  \ar[r]& Sym^3(\sE')\otimes \det(\sE')^{-1} \ar[r] & 0.
	} \tag{$++$}
    \end{equation}

	Since $\sM\in Pic^0(A')$ is nontrivial, we know that $H^i(M)=0$ for all $i$,  and therefore

	$$ (i) \qquad H^0( Sym^3(V^\vee)\otimes \hol_{A'}(M))=0.$$
	
	Furthermore, since $D$ is ample, we have, by Kodaira vanishing, that 
	
	$$ (ii) \qquad H^1(Sym^2(V^\vee)\otimes \hol_{A'}(-D+2M))=0.$$
	
	Finally, relations $(i)$ and $(ii)$, and  the long exact cohomology sequence   associated to \eqref{twisted-eag-north}, imply that
	
	\[
	H^0(Sym^3(\sE')\otimes \det(\sE')^{-1})=0.
	\]
	
	The conclusion  is then the following
	
	\begin{thm}
		The case $\delta=d=3$, $p_g=q=2$  occurs  under the Generality Assumption \ref{GenAss} exactly 
		 for the AC3 family.
		
		That is, all  minimal surfaces $S$ of general type with $p_g=q=2$, $K^2=6$, with Albanese map of degree $3$ and satisfying the Generality Assumption belong to the  family described in  Subsection \ref{new-family}.

		 Moreover, the only other minimal surfaces $S$ of general type with $p_g=q$, $K^2=6$, with $\al :S  \ra A$ a surjective morphism
		 of degree $d=3$ onto an Abelian surface, and satisfying the Generality Assumption,
		 are the surfaces with $p_g=q=3$ described in Subsection \ref{Hesse-cubic}.
	\end{thm}

   \bigskip

 \bigskip

\bigskip

\section{The case $d=4$  under the generality  assumption:  an example with  $d=\de=4$ and
 nonzero homogeneous bundle $\mathfrak H$.} \label{d=de=4}

In this section, we have that $\al : S \ra A$ has degree $d=4$, hence $\sE, \sE'$ have rank equal to $3$ 
and moreover there is a Heisenberg-equivariant exact sequence \eqref{MainSeqGenAss}, namely
\begin{equation*}
	0 \ra \mathfrak H' \ra \sL \otimes V^\vee   \ra  \sE'\ra 0.
\end{equation*}
Since $\sL \otimes V^\vee$ has rank $\de$, $\mathfrak H'$ has rank $\de - 3$, and by definition
it is a succesive extension of line bundles in $Pic^0(A')$.

By Casnati-Ekedahl we must have either 
\begin{enumerate}
	\item $ \det(\sF) = \det (\sE')= \de D $, or 
	
	\item  $ \det(\sF) = \det (\sE')= \de D - M$,  for $M$ a nontrivial line bundle in $Pic^0(A')$.
\end{enumerate}

Again by \cite{c-e} we have the following formula
 $$c_2(\sF) = \de^2 \ K_S^2   -2 c_1(\sE')^2 + 4 c_2(\sE')) = \de^2 ( \ K_S^2 - 4).$$

In the first case a  natural choice is to take $\sF= \hol_{A'}(2D) \oplus \hol_{A'}(2D)$,
and to take the inclusion  $\sF \subset S^2(\sE')$ as induced by a two dimensional
subspace of $S^2( V^\vee )$. We have then, by the above formula, $\de=4, K_S^2=6$.

\subsection{The case $d=\de=4$}

 We shall show that this case  
	occurs for a polarization  of type $(1,4)$, but without
	yielding   $p_g=q=2$.
	
	Here,  since $\de=4$,  $ \mathfrak H'$ is a line bundle in $Pic^0(A')$.

There are again two cases: 
\begin{enumerate}
	\item $ \mathfrak H' = \hol_{A'}$,
	\item $ \mathfrak H' = \hol_{A'}(M)$, with $M$ nontrivial in $Pic^0(A')$.
\end{enumerate}

Consider  the first case with $\sF= \hol_{A'}(2D) \oplus \hol_{A'}(2D)$. Then, by the formulae of \cite{c-e},
not only $K_S^2=6$, but also $\chi(S) = 1$.

Here $S' \subset \PP(V) \times A'$ is the complete intersection of three divisors of respective classes
$2H, 2H, H + D$ and more precisely 
$$ S' = \{ (y,z) \in \PP(V) \times A' \ | \ Q_1(y)= Q_2(y)= \sum_{j=1}^4 y_j x_j(z)=0\},$$
where $x_1, \dotso, x_4$ is the canonical basis of $V = H^0(A', \hol_{A'}(D))$,
 $y_1, \dotso, y_4$ the dual basis of $V^\vee$, 
and the subspace generated by the quadrics $Q_1(y), \ Q_2(y)$ is Heisenberg invariant.

Hence, $S' \subset C \times A'$, where $C$ is an elliptic normal quartic which is Heisenberg invariant.

We see immediately that the polarization $D$ is of type $(1,4)$ since for type $(2,2)$
we would have $G = (\ZZ/2)^4$ acting faithfully on $\PP^3=\PP(V)$, while there is no faithful action
of $ (\ZZ/2)^4$ on an elliptic curve $C$.

By classical formulae, see  \cite{hulek}, page 28, we have that by Heisenberg invariance
the two quadratic equations are:

\begin{equation}
	\begin{split}
		Q_1(y) : = y_1^2 + y_3^2 + 2 \la y_2 y_4=0, \\
		Q_2(y) : = y_2^2 + y_4^2 + 2 \la y_1 y_3=0,
	\end{split}
	 \qquad \la \neq 0, \pm 1, \pm i.
\end{equation}

  The group  $G = (\ZZ/4)^2$ acts by translations on the normal elliptic curve $C$ of degree $4$, 
$$ C : =  \{ y   \in \PP^3 = \PP(V) \ | \  Q_1(y) = Q_2(y)  =0   \},$$
hence the quotient surface  $S: = S' /G$ has $q(S)=3$ and $\chi(S)=1, K_S^2=6$.

 There remains to show that, for a general choice of $\la \in \CC$, $S'$ is smooth.

To this purpose we apply the Theorem of Bertini-Sard, and we need only to show 
that the singular locus of 
$$\sS : =  \{ (y,z,\la ) \in \PP(V) \times A' \times \CC \ | \  y_1^2 + y_3^2 + 2 \la y_2 y_4 = y_2^2 + y_4^2 + 2 \la y_1 y_3= \sum_{j=1}^4 y_j x_j(z)=0\}$$
does not map surjectively onto the complex line $\CC$ 	with coordinate $\la$.

We define $\sC \subset \PP^3 \times \CC $ as the family of normal elliptic quartics given by
the	 zero set
\begin{equation}
	\begin{cases}
		Q_1(y,\la):= y_1^2 + y_3^2 + 2 \la y_2 y_4 =0 \\
		Q_2(y,\la):= y_2^2 + y_4^2 + 2 \la y_1 y_3=0
	\end{cases}
\end{equation}

\begin{remark}
	$\sC$ is birational to a smooth quartic surface $\sC'$ in $\PP^3$, defined by the following equation 
	\begin{equation*}
		(y_1^2+y_3^2)y_1y_3= (y_2^2+y_4^2)y_2y_4,
	\end{equation*}
	since 
	\[
	Q_1(y,\la)=Q_2(y,\la)=0 \quad \iff \quad -2\la = \frac{y_1^2+y_3^2}{y_2y_4}=\frac{y_2^2+y_4^2}{y_1y_3}.
	\]
	Hence, $\sS$ is a hypersurface in the $4$-fold $\sC \times A'$.
\end{remark}

We observe that 
\[
\Sing(\sS)=\{(y,\la, z) \in \sC\times A' \ | \ z \in \Sing(D_y) \text{ and } \rank (M) =2\},
\]
where $D_y:=\{z\in A' \ | \ \sum_j y_jx_j(z)=0\}$ and 
\[
M:=
\begin{pmatrix}
	2y_1	& 2\la y_4	& 2y_3	& 2\la y_2	& 2y_2y_4 \\
	2\la y_3	& 2 y_2	& 2 \la y_1	& 2 y_4	& 2y_1y_3 \\
	x_1		& x_2	&	x_3& 	x_4&	0
\end{pmatrix}.
\]
Since for  $\la \neq 0, \pm 1, \pm i$ the curve $\{Q_1(y,\la)=Q_2(y,\la)=0\}$ is smooth, the first two rows $M_1,M_2$ of $M$ are
then  linearly independent and then, on an open set of $\sC$,  it must hold
\[
x=y_1y_3\cdot M_1 - y_2y_4 \cdot M_2.
\]
Hence, writing $x$ as a column, we have 
\[
x=
\begin{pmatrix}
	y_1^2y_3 - \la y_2y_3y_4 \\
	-y_2^2y_4 + \la y_1y_3y_4\\
	y_1y_3^2 - \la y_1y_2y_4\\
	-y_2y_4^2 +\la y_1y_2y_3
\end{pmatrix}=:	\beta(y,\la)
\]
That is, we have  a rational map
\[
\beta \colon \sC \dashrightarrow \PP^3
\]
such that 
\[
x=\beta(y,\la)=\beta_0(y)+\la \beta_1(y),
\]
where clearly
\[
\beta_0(y):={^t}(y_1^2y_3 , \ -y_2^2y_4, \   y_1y_3^2 , \  -y_2y_4^2),
\]
\[
\beta_1(y):={^t}(-  y_2y_3y_4, \  y_1y_3y_4, \ -  y_1y_2y_4, \  y_1y_2y_3).
\]

Recalling that
\[
-2\la = \frac{y_1^2+y_3^2}{y_2y_4}=\frac{y_2^2+y_4^2}{y_1y_3},
\]
we can write
\[
x=2 \beta_0(y) - \frac{y_1^2+y_3^2}{y_2y_4} \beta_1(y)= 2y_2y_4 \beta_0(y)  - (y_1^2+y_3^2) \beta_1(y),
\]
and this shows that the rational map $\beta$ is given by homogeneous polynomials of degree $5$.

Recall that, for a general pair $(A',D)$, the image $\Sigma$ of the map associated to the linear system $|D|$, namely
\[
x\colon A' \ra \Sigma \subset \PP^3, \qquad z \mapsto (x_1(z), \ x_2(z), \ x_3(z),  \ x_4(z)),
\]
is an octic surface in $\PP^3$ whose equation depends on some $c=(c_0, \dotso, c_3) \in \PP^3$ (see \cite{b-l-s}).

Let $\Delta\subset \PP^3$ be the discriminant of the linear system $|D|$, namely
\[
\Delta:=\lbrace y \ | \ D_y \text{ is singular} \rbrace.
\]
Hence, we define the following divisors in $\sC$
\[
N_1 :=  \beta^{-1}(\Sigma), \qquad N_2:=\sC \cap \Delta 
\]
($N_2$ is the birational inverse image of $  \sC' \cap \Delta$).

Moreover, we have the equation  (asserting that $x = \beta(y,\la)$ belongs to the plane $y^{\perp}$)
\[
y \cdot \beta(y,\la)=0, 
\]
defining
\[
N_3:=\{(y,\la) \in \sC \ | \ y \cdot \beta(y,\la)=0 \} \subset \sC.
\]

\begin{remark}
	A straightforward computation shows that actually $N_3 = \sC$. In fact:
	\[
	y \cdot \beta(y,\la)=0 \iff y_1^3y_3- y_2^3y_4 + y_1y_3^3 - y_2y_4^3=0 \iff 	(y_1^2+y_3^2)y_1y_3= (y_2^2+y_4^2)y_2y_4.
	\]
\end{remark}

Therefore, $\Sing(\sS)$ does not map onto $\CC$ if
\[
| N_1 \cap N_2| < \infty.
\] 
\begin{remark}
	(a) Certainly, $N_2$ is a curve, since the surfaces $\Sigma$ vary, hence their discriminants, and $\sC$, $\Sigma$ are irreducible. By a similar argument also $N_1$ is a curve in $\sC$ ($\Sigma$ moves).
	
	(b) We should also impose the condition that $x^{\perp}$ is	 tangent to $\Delta$ at $y$.
	
	Since $\Delta$ is the dual surface to $\Sigma$, this condition means that $y^{\perp}$ is tangent to $\Sigma$ at $x$.
\end{remark}

It suffices in any case to   show that $N_1\cap N_2$ is a finite set for a general $A'$.

Let $F(c,x)=0$, for $c \in \PP^3, x \in \PP^3$, be   the equation of the family of octics $ \Sigma_c $, given   in \cite{b-l-s};
let $ p(y) =0, \ y \in \PP^3,$ be the equation for the surface  $ \sC'$.

Since $\Delta$ is the dual surface of  $ \Sigma$, we denote by $ \nabla F$ the gradient with respect 
to the variables $ x$, and we set
$$y = \nabla F (c,x) .$$

Consider then the three equations
\begin{equation} \label{script}
	\begin{split}
		&F(c,x)=0,  \\
		&p' : = p(  \nabla F (c,x) )=0,\\
		&F':= F(c, \beta (  \nabla F (c,x) ) = 0,
	\end{split}
	\tag{$\star \star \star$}
\end{equation}
which, for a general $c\in \PP^3$, describe the set $N_1\cap N_2$.

Since our aim is to show that, for a general $c$, we have a	 finite number of solutions, we 
view the equations \eqref{script} as equations on  $  \PP^3 \times \PP^3$, describing
$$ W : = \{ (c,x) \in  \PP^3 \times \PP^3 | F(c,x)= p'(c,x)= F'(c,x)=0\}.$$
Hence, our claim is equivalent to showing that the components of $W$ which  dominate the 
$ \PP^3$ with coordinates $c$ have 
 dimension $3$.
 
We therefore need to calculate the Jacobian matrix of the vector valued function $(F, p',F')$,
and show that this matrix has generically rank equal to three.

This should  be done by a computer algebra program.

\bigskip

\bigskip

\bigskip

{\bf Acknowledgement:} Thanks to Matteo Penegini for asking whether  a simple construction
exists  in the case of CHPP surfaces for which  the Albanese map (of $S'$)  is a Galois covering of degree $3$,
and some useful discussion.

The first named author would also like to thank him for support and encouragement.

He also gratefully acknowledges partial financial support by PRIN 2020KKWT53 003 - Progetto \em{Curves, Ricci flat Varieties and their Interactions}.

\end{document}